\documentclass{amsart}

\usepackage{amsmath,amssymb}
\usepackage{mathtools}
\usepackage{bm}
\usepackage{theoremref}
\usepackage{caption}
\usepackage{subcaption}
\usepackage{tikz}

\usepackage{graphicx}
\usepackage[margin=3.8cm]{geometry}

\newtheorem{theorem}{Theorem}[section]
\newtheorem{lemma}[theorem]{Lemma}
\newtheorem{corollary}[theorem]{Corollary}

\newtheorem{proposition}[theorem]{Proposition}
\newtheorem{claim}[theorem]{Claim}

\theoremstyle{definition}

\theoremstyle{remark}
\newtheorem{remark}{Remark}

\def\cB{\mathcal{B}}
\def\cP{\mathcal{P}}
\def\cS{\mathcal{S}}
\def\tw{\widetilde{w}}
\def\hQ{\widehat{Q}}

\def\tse{t^{\searrow}}
\def\tnw{t^{\nwarrow}}

\def\opi{\overline{\pi}}

\def\bx{\bar{x}}
\def\by{\bar{y}}

\begin{document}

 \title[On the enumeration of plane bipolar posets and transversal structures]{On the enumeration of plane bipolar posets and transversal structures}



\author[\'E. Fusy]{\'Eric Fusy}
\address{\'E.F.: CNRS/LIGM, UMR 8049, Universit\'e Gustave Eiffel, Marne-la-vall\'ee, France.}
\email{eric.fusy@u-pem.fr}

\author[E. Narmanli]{Erkan Narmanli}
\address{E.N.: LIX, \'Ecole Polytechnique, Palaiseau, France.}
\email{erkan.narmanli@normalesup.org}

\author[G. Schaeffer]{Gilles Schaeffer}
\address{G.S.: CNRS/LIX, UMR 7161, \'Ecole Polytechnique, Palaiseau, France.}
\email{schaeffe@lix.polytechnique.fr}



\maketitle

\begin{abstract}
We show that plane bipolar posets (i.e., plane bipolar orientations with no transitive edge) and transversal structures can be set in correspondence to certain (weighted) models of quadrant walks, via suitable specializations of a bijection due to Kenyon, Miller, Sheffield and Wilson. 
We then derive exact and asymptotic counting results. 
In particular we prove (computationally and then bijectively) that the number of plane bipolar posets on $n+2$ vertices equals the number of plane permutations (i.e., avoiding the vincular pattern $2\underbracket[.5pt][1pt]{14}3$) of size $n$. 
 Regarding transversal structures, for each $v\geq 0$ we consider $t_n(v)$ the number of such structures with $n+4$ vertices and weight $v$ per quadrangular inner face 
 (the case $v=0$ corresponds to having only triangular inner faces). We obtain a recurrence to compute $t_n(v)$, and an asymptotic formula that for $v=0$ 
 gives $t_n(0)\sim c\ \!(27/2)^nn^{-1-\pi/\mathrm{arccos}(7/8)}$ for some $c>0$, which also ensures  that the associated generating function is not D-finite.
\end{abstract}


\section{Introduction}
The combinatorics of planar maps (i.e., planar multigraphs endowed with an embedding on the sphere) 
has been a very active research topic ever since the early works of W.T. Tutte~\cite{Tutte63}. In the last few years, after tremendous progresses on the enumerative and probabilistic theory of maps~\cite{Bous11,bona15,ey16,Mi14}, the focus has started to shift to planar maps endowed with \emph{constrained orientations}. Indeed constrained orientations capture a rich variety of models~\cite{felsner2004lattice,eppsteinbook} with  connections to (among others) graph drawing~\cite{schnyder1990embedding,bonichon2007convex}, pattern-avoiding permutations~\cite{bonichon2010baxter,reading2012generic} and the study of their permuton limits~\cite{borga2020scaling,borga2021skew,borga2021permuton}, 
Liouville quantum gravity~\cite{li2017schnyder}, or theoretical physics~\cite{loll2015locally}. From an enumerative perspective, these new families of maps are expected to depart (e.g. \cite{felsner2010asymptotic,elveyprice:hal-01738160}) from the usual \emph{algebraic generating function} pattern followed by many families of planar maps with local constraints~\cite{schaeffer-book}. From a probabilistic point of view, they lead to new models of random graphs and surfaces, as opposed to the universal Brownian map limit capturing earlier models. Both phenomena are first witnessed by the appearance of new critical exponents $\alpha\neq5/2$ in the generic $\gamma^nn^{-\alpha}$ asymptotic formulas for the number of maps of size~$n$, e.g. the critical exponent $\alpha=4$ holds for plane bipolar orientations~\cite{kenyon2019bipolar,bousquet2020plane}, whereas the exponent $\alpha=5/2$ is known to be universal~\cite{drmota2022universal} for undecorated map enumeration.

A fruitful approach to oriented planar maps is through bijections (e.g.~\cite{BERNARDI200955}) with walks with a specific step-set in the quadrant, or in a cone, up to shear transformations. We rely here on a recent such bijection due to Kenyon, Miller, Sheffield, and Wilson~\cite{kenyon2019bipolar}, 
hereafter referred to as the KMSW bijection, that encodes plane bipolar orientations by certain quadrant walks called \emph{tandem walks}, and that was recently used in the article~\cite{bousquet2020plane} to obtain counting formulas for plane bipolar orientations with control on the face-degrees. 
Precisely, we will exploit the KMSW bijection to study the enumerative properties of two families of oriented planar maps. On one hand, \emph{plane bipolar posets}, which are plane bipolar orientations with no transitive edges. These also correspond to so-called \emph{planar lattices} (i.e., lattices admitting a planar upward drawing) equipped with a planar embedding so that the min and max are both in the outer face.  Planar lattices have been extensively studied, in particular a classical result is that they are exactly the lattices of Dushnik-Miller dimension at most $2$ (see~\cite{kelly1982dimension} and references therein), and an analog of the Kuratowski theorem characterizing planarity of lattices has also been established~\cite{kelly1975planar}.   On the other hand, \emph{transversal structures}~\cite{he1993finding,kant1997regular,fusy2009transversal} (also known as regular edge labelings) are specific partitions of the inner edges of maps (with a quadrangular outer face) into two plane bipolar posets that cross at any inner vertex. These are the topological embedded structures dual to rectangular tilings (as illustrated in Figures~\ref{fig:rectangular} and~\ref{fig:rectangular_quad}).  

\medskip

\noindent{{\bf Overview of the results}.} We show in Section~\ref{sec:corresp} that the KMSW bijection can be adapted to plane bipolar posets and transversal structures, by a suitable reduction of these models to plane bipolar orientations with some decorations on the faces.    
Building on these specializations, in Section~\ref{sec:exac} we obtain exact enumeration results. In particular, we show that the number $b_n$ of plane bipolar posets on $n+2$ vertices is equal to the number of plane permutations of size $n$ introduced in~\cite{bousquet2007forest} and recently further studied in~\cite{bouvel2018semi}, and that a reduction to small-steps quadrant walks models (which makes coefficient computation faster) can be performed for the number $e_n$ of plane bipolar posets with $n$ edges, and for the number $t_n(v)$ of transversal structures on $n+4$ vertices
with weight $v$ per quadrangular inner face. 
 In Section~\ref{sec:asymptotic} we then obtain asymptotic formulas for the coefficients $b_n,e_n,t_n(v)$ (with $v\geq 0$), 
 all of the form $c\gamma^nn^{-\alpha}$ with $c>0$ and with $\gamma,\alpha\neq5/2$ explicit. Using the approach of~\cite{bostan2012} we then deduce from these estimates that the generating functions for $e_n, t_n(0), t_n(1)$ are not D-finite. We also briefly explain that, based on our asymptotic estimates on $t_n(v)$, random 
 transversal structures on $n+4$ vertices with weight $v$ 
 per quadrangular face can be proposed as a model that interpolates between a random lattice (in the regime $v=\Theta(1)$) and a regular lattice (in the regime $v\to\infty$).  
 Finally, in Section~\ref{sec:bijec} we provide a direct bijection between plane permutations of size $n$ and plane bipolar posets with $n+2$ vertices, which is similar to the one~\cite{bonichon2010baxter} between Baxter permutations and plane bipolar orientations.

\section{Oriented planar maps and tandem walks in the quadrant}\label{sec:corresp}
After general definitions and properties on plane bipolar orientations in Section~\ref{sec:bipo}, 
we recall in Section~\ref{sec:KMSW} the KMSW bijection, which encodes plane bipolar orientations by certain quadrant walks, while controlling the number of edges, vertices, and faces of each type.
Then, respectively in Section~\ref{sec:appli_KMSW} and Section~\ref{sec:appli_ts}, 
we show how the bijection can be applied
to plane bipolar posets (in two different ways) and to transversal structures.

\subsection{Plane bipolar orientations}\label{sec:bipo}
A \emph{planar map} (we refer to the survey~\cite{schaeffer-book}) 
is a connected multigraph embedded on the oriented sphere. A \emph{rooted planar map} is a planar map with a marked corner, called its \emph{root}.  
The \emph{root vertex} (resp. \emph{root face}) is the vertex (resp. face) incident to the root,
the root face is taken as the unbounded  face in planar representations (obtained by a projection from the root face). 

In an orientation of a planar map, a corner $c=(v,e_1,e_2)$ is called \emph{lateral} if one of $e_1,e_2$ is ingoing at $v$ and the other one is outgoing at $v$, it is called \emph{extremal} otherwise, i.e., if either $e_1,e_2$ are both ingoing or both outgoing at $v$. 

A \emph{plane bipolar orientation} is a rooted planar map endowed with an acyclic orientation with a unique source $S$ at the root-vertex, and a unique sink $N$ incident to the outer face. The vertices $S,N$ are called the \emph{poles} of the orientation.  
It is known~\cite{de1995bipolar} that a plane bipolar orientation is characterized by the following local properties (for orientations with $S$ as a  
source and $N$ as a sink), illustrated in Figure~\ref{fig:bipolar_rules}:
\begin{itemize} 
\item[(B):] Each non-pole vertex has two lateral corners (so the incident edges form two groups: ingoing and outgoing edges).
\item[(B'):] Each face $f$ (including the outer one) has two extremal corners, so that the contour is partitioned into a left lateral path $L_f$ and a right lateral path $R_f$ that share their origins and ends, which are called the \emph{bottom vertex} and \emph{top vertex} of the face.
\end{itemize}
The \emph{type} of a face $f$  is the pair $(i,j)$ where $i+1$ (resp. $j+1$) is the length of $L_f$ (resp. $R_f$). The \emph{outer type} of a plane bipolar orientation 
is the type of the outer face, while the \emph{pole-type} is the pair $(p,q)$ such that $p+1$ is the degree of $S$ and $q+1$ is the degree of~$N$.

\begin{figure}
\begin{center}
\includegraphics[width=0.8\textwidth]{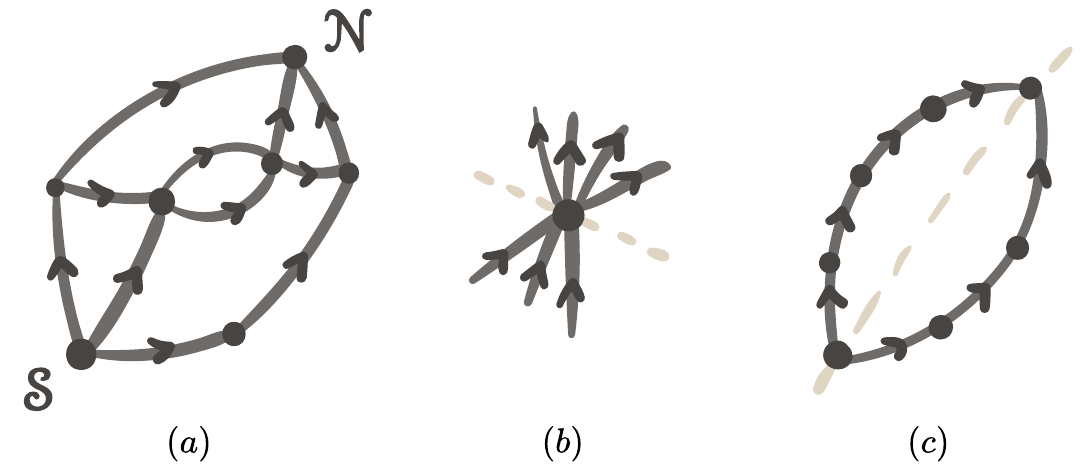}
\end{center}
\caption{(a) A plane bipolar orientation. (b) The local conditions at non-pole vertices. (c) The local condition at inner faces.}
\label{fig:bipolar_rules}
\end{figure}

\begin{figure}
\begin{center}
\includegraphics[width=\textwidth]{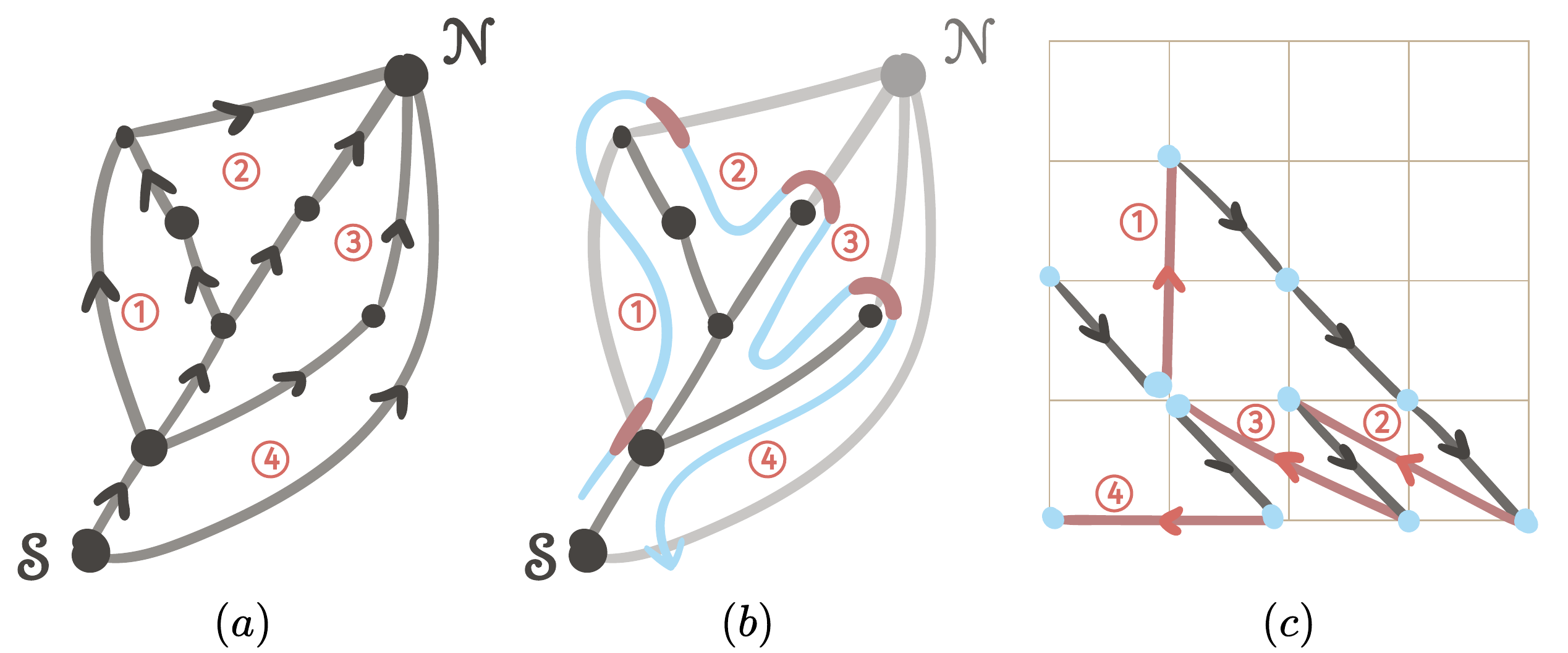}
\end{center}
\caption{A plane bipolar orientation of outer type $(2,0)$, and the corresponding quadrant tandem walk from $(0,2)$ to $(0,0)$ under the KMSW bijection.}
\label{fig:ex_kmsw}
\end{figure}
 
\subsection{The KMSW bijection}\label{sec:KMSW}

A \emph{tandem walk} is a walk on $\mathbb{Z}^2$ with steps in~$\{(1,-1)\}\cup\{(-i,j), i,j\geq 0\}$; it is a \emph{quadrant} walk if it stays in $\mathbb{N}^2$ all along. 
We use the usual terminology\footnote{Similarly later on, we will usually use compass notation for short steps.} of SE steps for the steps $(1,-1)$.  
On the other hand, every step $(-i,j)$ in such a walk is called a \emph{face-step}. 
The \emph{KMSW bijection} maps a plane bipolar orientation $B$ to a tandem walk staying in the quadrant.   
Letting $(a,b)$ be the outer type of $B$, the tandem walk starts at $(0,a)$ and is drawn step by step while traversing $B$, by the following procedure. 
We consider the \emph{rightmost ingoing tree} for $B$, i.e., the tree $T$ spanning all vertices except $N$ and rooted at $S$, where the parent-edge of every non-pole vertex
is its rightmost ingoing edge. Then, during a clockwise walk  around $T$ starting at the root (at $S$),  we draw a SE step when walking along a tree-edge away from the root,  
and we draw a face-step $(-i,j)$ when first entering an inner face of type $(i,j)$, see Figure~\ref{fig:ex_kmsw} for an example. 

\begin{theorem}[\cite{kenyon2019bipolar}]\label{theo:KMSW}
The KMSW bijection is a bijection between plane bipolar orientations of outer type $(a,b)$ with $n+1$ edges, and quadrant tandem walks of length $n$ from $(0,a)$ to $(b,0)$. Every non-pole vertex corresponds to a SE step, and every inner face of type $(i,j)$ corresponds to a face-step $(-i,j)$. 
\end{theorem}



\subsection{Application to plane bipolar posets}\label{sec:appli_KMSW}

For $B$ a plane bipolar orientation, an edge $e=(u,v)\in B$ is called \emph{transitive} if there is a path from $u$ to $v$ avoiding $e$.
If $B$ has no transitive edge it is called a \emph{plane bipolar poset}. 

\begin{claim}\label{claim:tra}
Let $B$ be a plane bipolar orientation. Then $B$ is a plane bipolar poset iff for each inner face, its type $(i,j)$ satisfies $i\geq 1$ and $j\geq 1$. 
\end{claim}
\begin{proof}
Assume $B$ has such an inner face $f$, and w.l.o.g. assume that $L_f$ has length $1$. Then $L_f$ is an edge $e=(u,v)$, and it is transitive since $R_f$ avoids $e$ and goes from $u$ to $v$. Conversely, assume $B$ has a transitive edge $e=(u,v)$, and let $P\neq{e}$ be a path from $u$ to $v$ such that the region enclosed by $P+e$ is minimal. Then, by minimality, the interior of $P+e$ has to be a face, which has $e$ as one of its lateral paths. 
\end{proof}

A tandem walk is called \emph{$E$-admissible} if every face-step $(-i,j)$ has $i\geq 1$ and $j\geq 1$. 

\begin{proposition}\label{prop:E-admissible}
The KMSW bijection specializes into a bijection between plane bipolar posets of outer type $(a,b)$, and $E$-admissible  tandem walks from $(0,a)$ to $(b,0)$ in the quadrant. 
\end{proposition}
\begin{proof}
It is a direct consequence of Theorem~\ref{theo:KMSW} and Claim~\ref{claim:tra}. 
Indeed, Claim~\ref{claim:tra} ensures that a plane bipolar orientation is a plane bipolar poset iff the type $(i,j)$ of every inner face satisfies $i\geq 1$ and $j\geq 1$, which 
occurs iff every face-step $(-i,j)$ in the corresponding tandem walk  satisfies $i\geq 1$ and $j\geq 1$, i.e., the tandem walk is $E$-admissible. 
\end{proof}

Note that in Proposition~\ref{prop:E-admissible}, the primary parameter of the poset (the one corresponding to the walk length) is 
the number of edges (minus $1$). 
We now give a bijection that will allow us to encode plane bipolar posets by (weighted)  quadrant tandem walks, this time with the number of vertices as the primary parameter (associated wtih the walk length)\footnote{Though we will focus on univariate enumeration of plane bipolar posets, both approaches make it possible to compute the number $b_{n,m}$ of plane bipolar posets having $n$ vertices and $m$ edges.}. 

Given  a plane bipolar orientation~$B$ (whose vertices and edges are black), a~\emph{transversal completion} of~$B$ consists of the following steps:
\begin{itemize}
	\item subdivide every edge into a path of length 2, with a so-called~\emph{subdivision-vertex} (white vertex) in the middle ;
	\item in every inner face~$f$, add so-called~\emph{transversal edges} (red edges), each such edge from a subdivision-vertex on the left side to a subdivision-vertex on the right side of~$f$, so that after adding these edges, all faces within~$f$ are of one of the 4 types shown in~Figure~\ref{fig:face_v-transverse}. Such a configuration of edges within $f$ 
	is called \emph{admissible}.
\end{itemize}
\begin{figure}
\begin{center}
\includegraphics[width=\textwidth]{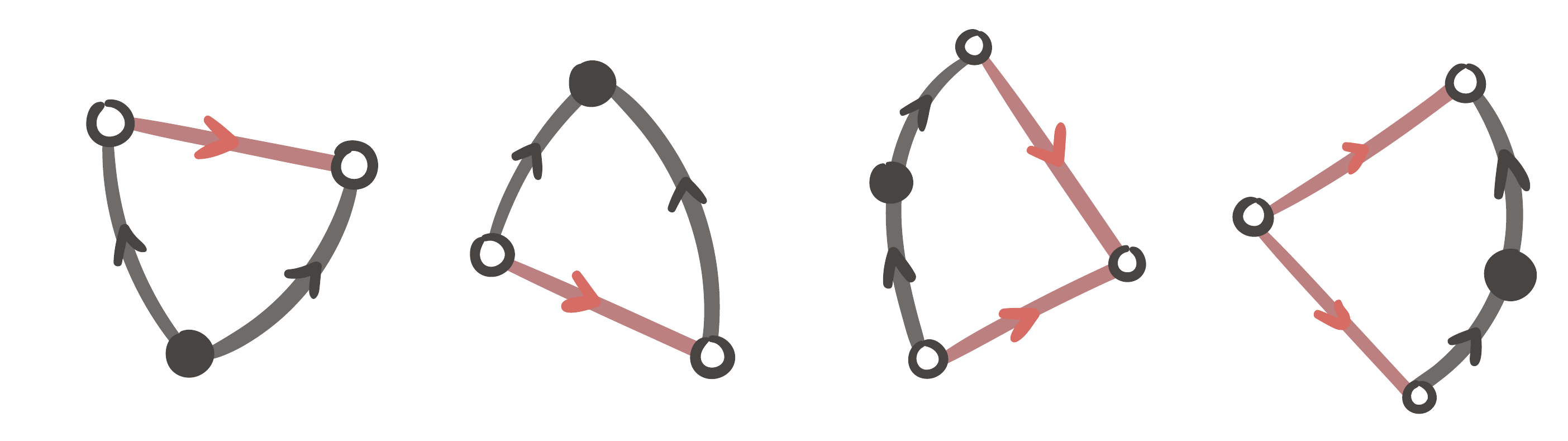}
\end{center}
\caption{The four possible types of inner faces in a~$V$-transverse bipolar orientation.
}
\label{fig:face_v-transverse}
\end{figure}
A~\emph{$V$-transverse bipolar orientation} is defined as a plane bipolar orientation endowed with a transversal completion, see the left-part of 
Figure~\ref{fig:bij_bipolar_pos} for an example.

\begin{lemma}\thlabel{res:transversal_completion}
	Let~$B$	be a plane bipolar orientation, and let~$f$ be an inner face of type~$(i,j)$ in $B$. 
	Then, the admissible configurations within $f$ ---in transversal completions of $B$--- 
	are encoded by walks from~$(0,0)$ to~$(-i,j)$ with steps in~$\{W,N\}$.
\end{lemma}
\begin{proof}
	The addition of edges within $f$ is done as follows (see Figure~\ref{fig:trans_completion}). 
	Given a walk $\gamma$ from $(0,0)$ to~$(-i,j)$ with steps in~$\{W,N\}$, we start with~$f$ having no transversal edge, and maintain two marked vertices :~$v_\ell$ on the left lateral path of~$f$ and~$v_r$ on its right lateral path. We start with~$v_\ell$ (resp.~$v_r$) on the first white vertex on the left (resp. right) lateral path, and we start by drawing a transversal edge directed from~$v_\ell$ toward~$v_r$. When reading a step~$W$ (resp.~$N$) of $\gamma$, we move up~$v_\ell$ (resp.~$v_r$) on the next white vertex on the left (resp. right) lateral path of~$f$; then we add a transversal edge directed from~$v_\ell$ toward~$v_r$. 
\begin{figure}
\begin{center}
\includegraphics[width=0.7\textwidth]{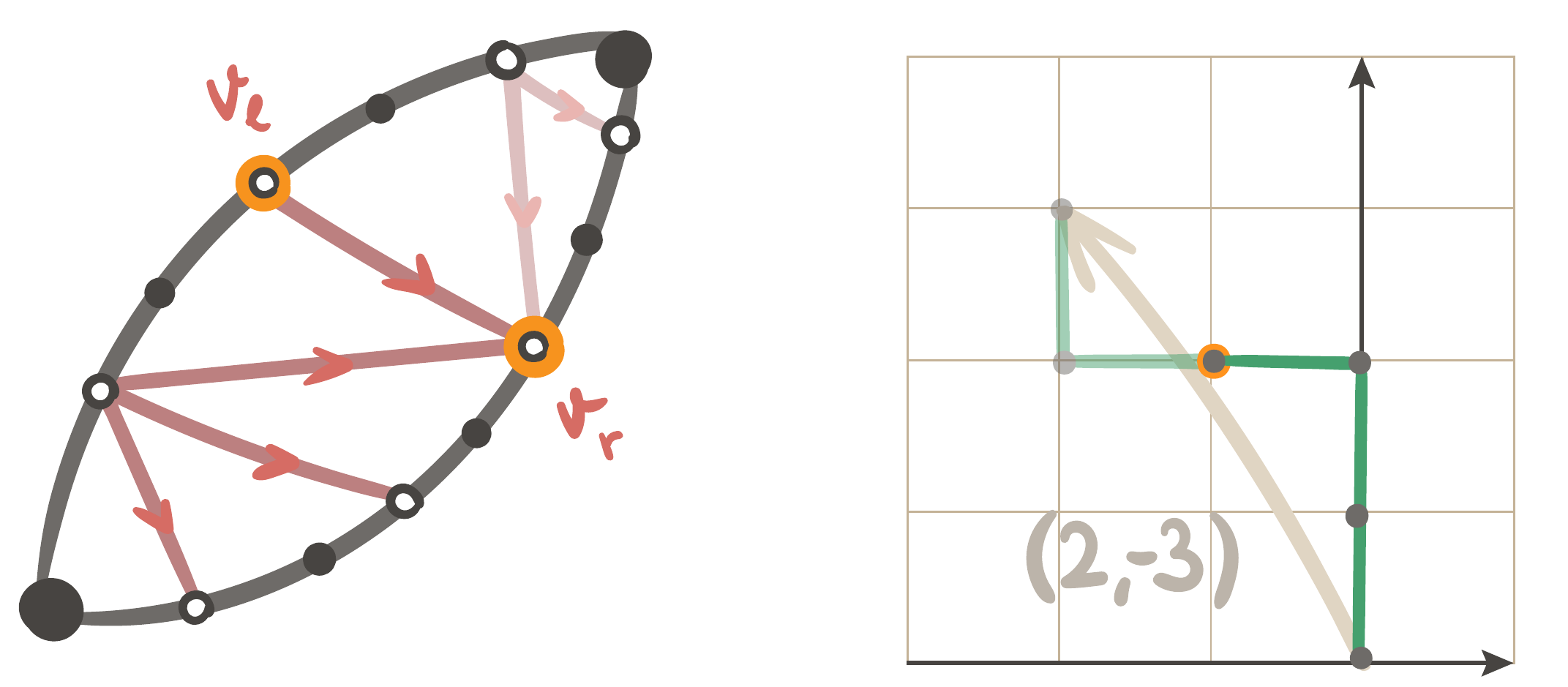}
\end{center}
\caption{A possible configuration of transversal edges within an inner face (of the underlying plane bipolar orientation) 
in a $V$-transverse bipolar orientation, and the corresponding walk with steps in~$\{W,N\}$.}
\label{fig:trans_completion}
\end{figure}

Conversely, every admissible configuration within $f$ clearly yields such a walk, 
by traversing the faces in $f$ from bottom to top, excluding the first and last one,
 and writing a step $W$ (resp. $N$) each time a face of the 3rd (resp. 4th) type is traversed.  	
\end{proof}
We define a \emph{$V$-admissible} tandem walk as a tandem walk where, to each face-step $(-i,j)$, a walk with steps in $\{W,N\}$ is attached, with same starting and ending point as the face-step. 
\begin{proposition}\label{prop:posets_vertices}
Plane bipolar posets of pole-type $(p,q)$, with $n+2$ vertices and $f$ inner faces, are in bijection with $V$-transverse bipolar orientations such that the underlying plane bipolar orientation has outer type $(p,q)$, with $n$ edges and $f+2$ vertices. These are in bijection  with $V$-admissible quadrant tandem walks 
 of length $n-1$ from $(0,p)$ to $(q,0)$ with $f$ SE steps.
\end{proposition}
\begin{remark}
Proposition~\ref{prop:posets_vertices} yields an extension of the bijection in~\cite{fusy2010new} between plane bipolar posets with no N-pattern and bipolar orientations.
That bijection corresponds to the case where the walk attached to each face-step $(-i,j)$ is $W^iN^j$, as illustrated in Figure~\ref{fig:poset_no_N}. \hfill$\triangle$
\end{remark}
\begin{figure}
\begin{center}
\includegraphics[width=0.7\textwidth]{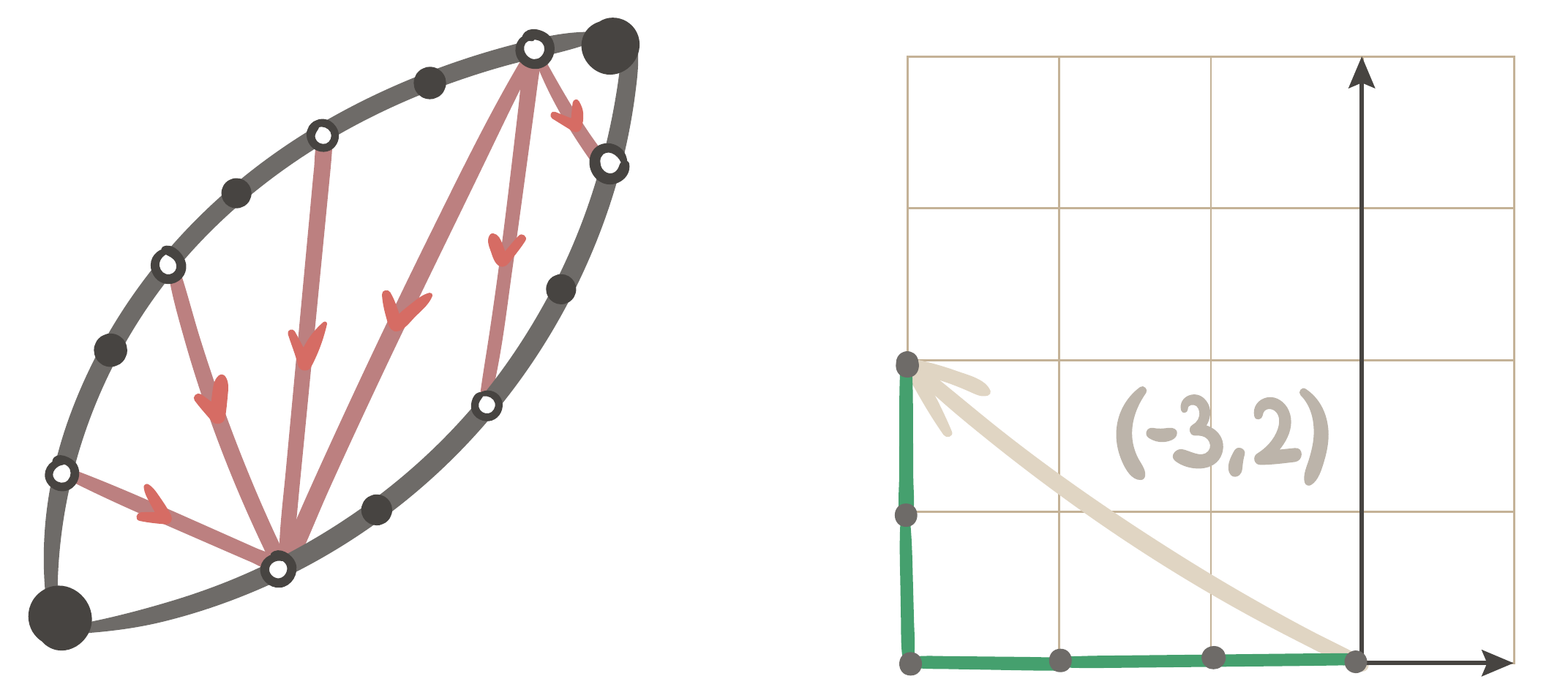}
\end{center}
\caption{The unique completion avoiding the appearance of an N-pattern is the one where the encoding walk is of the form $W^iN^j$.}
\label{fig:poset_no_N}
\end{figure}
\begin{proof}
Let~$B$ be a~$V$-transverse bipolar orientation with poles~$S'$ and~$N'$. Let~$B'$ be obtained as follows (see~Figure~\ref{fig:bij_bipolar_pos}):
  \begin{itemize}
	\item We add two vertices~$S$ and~$N$ in the outer face of~$B$, respectively next to the left lateral path and next to the right lateral path of $B$.
	\item For each white vertex $v$ on the left (resp. right) boundary of $B$, we add a red edge from $S$ to $v$ (resp. from $v$ to $N$).
	\end{itemize}
\begin{figure}
\begin{center}
\includegraphics[width=\textwidth]{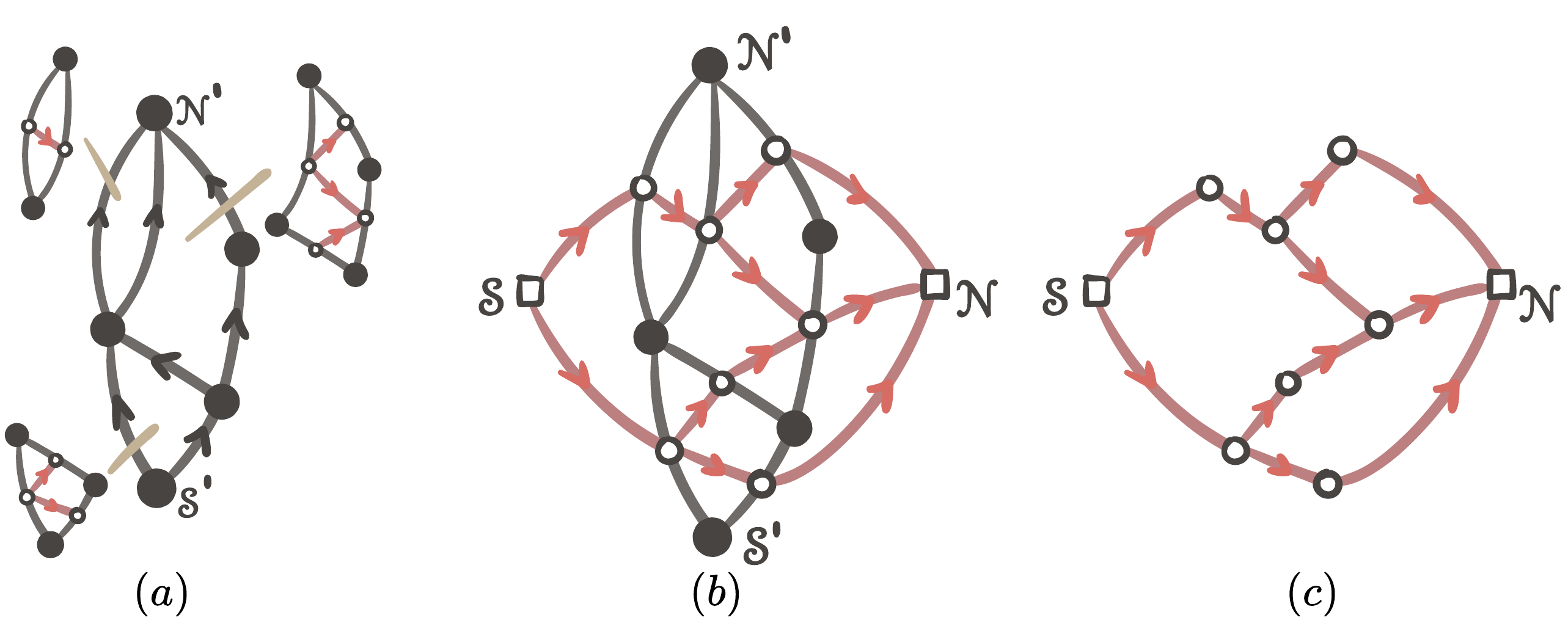}
\end{center}
\caption{(a) A $V$-transverse bipolar orientation~$B$, (b) the map~$B'$ obtained after addition of $S,N$ and their incident edges, (c) the plane bipolar poset~$P=\phi(B)$.}
\label{fig:bij_bipolar_pos}
\end{figure}	
	Let~$P$ be obtained from $B'$ by removing all back vertices and all black edges. Let $\Pi_E^*$ be the partial order~\cite[Proposition 5.1]{de1995bipolar} on the edges of $B$ 
	such that $e<e'$ if there is a sequence $e=e_1,\ldots,e_k=e'$ of edges such that for $1\leq i<k$, $e_i$ and $e_{i+1}$ are respectively on the left side and on the right side of  
	a same inner face. 
	Clearly, any directed edge of $P$ is from a white vertex in the middle of an edge $e\in B$ to a white vertex in the middle of an edge $e'\in B$ such that
	$e< e'$. Hence $P$ is acyclic. Moreover, if~$v$ is a white vertex different from~$N$ or~$S$, let $e$ be the edge of $B$ having $v$ in its middle, and let $f_\ell$ be the face
	of $B$ on the left of $e$ (possibly the ``left outer face"), and let $f_r$ be the face of $B$ on the right of $e$ (possibly the ``right outer face").  
	From Figure~\ref{fig:face_v-transverse} it easily follows that if $f_\ell$ is an inner face, then $v$ must have at least one ingoing edge within $f_{\ell}$, and if $f_{\ell}$
	is the left outer face, there is one edge from $S$ to $v$. Hence, $v$ has positive indegree. Similarly, it has positive outdegree.  Thus~$S$ is the only source 
	and $N$ the only sink of $P$. Hence,  $P$ is a plane bipolar orientation. 
	Moreover, from Figure~\ref{fig:face_v-transverse}, if follows that in $B'$ there is an inner face $f$ of $P$ around each non-pole black vertex $v$, with the configuration
	shown  in Figure~\ref{subfig:red_face_V-trans}. Thus, the type of $f$ is $(i,j)$ where $i$ (resp. $j$) is the outdegree (resp. indegree) of $v$. Hence, $P$ is a plane bipolar poset. 
	We let $\phi$ be the mapping that associates $P$ to $B$. 

\begin{figure}
\begin{center}
\begin{subfigure}[b]{0.4\textwidth}
         \centering
         \includegraphics[width=\textwidth]{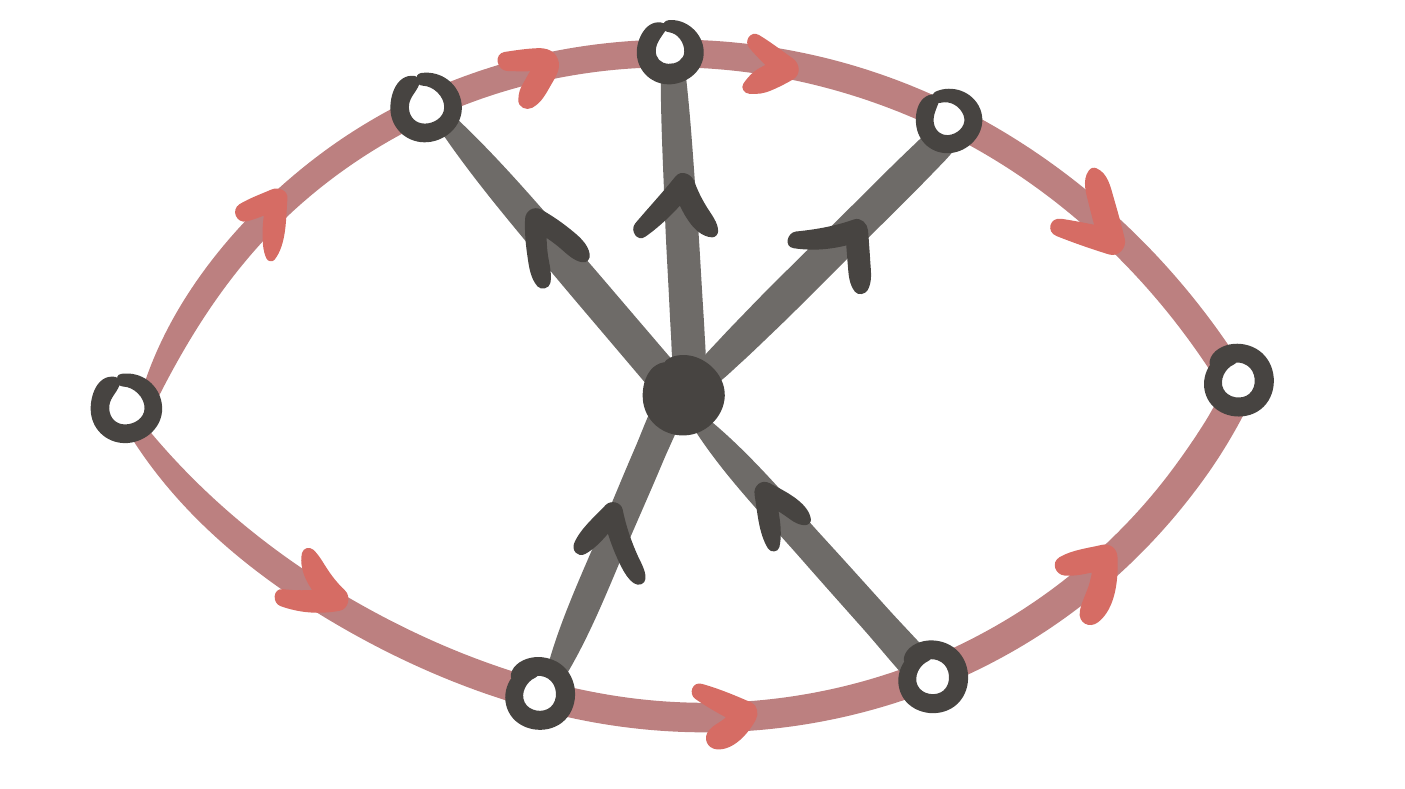}
    		\caption{Around a black vertex $\notin\{S',N'\}$.}
    		\label{subfig:red_face_V-trans}
     \end{subfigure}\hspace*{0.03\textwidth}%
         \begin{subfigure}[b]{0.4\textwidth}
         \centering
         \includegraphics[width=\textwidth]{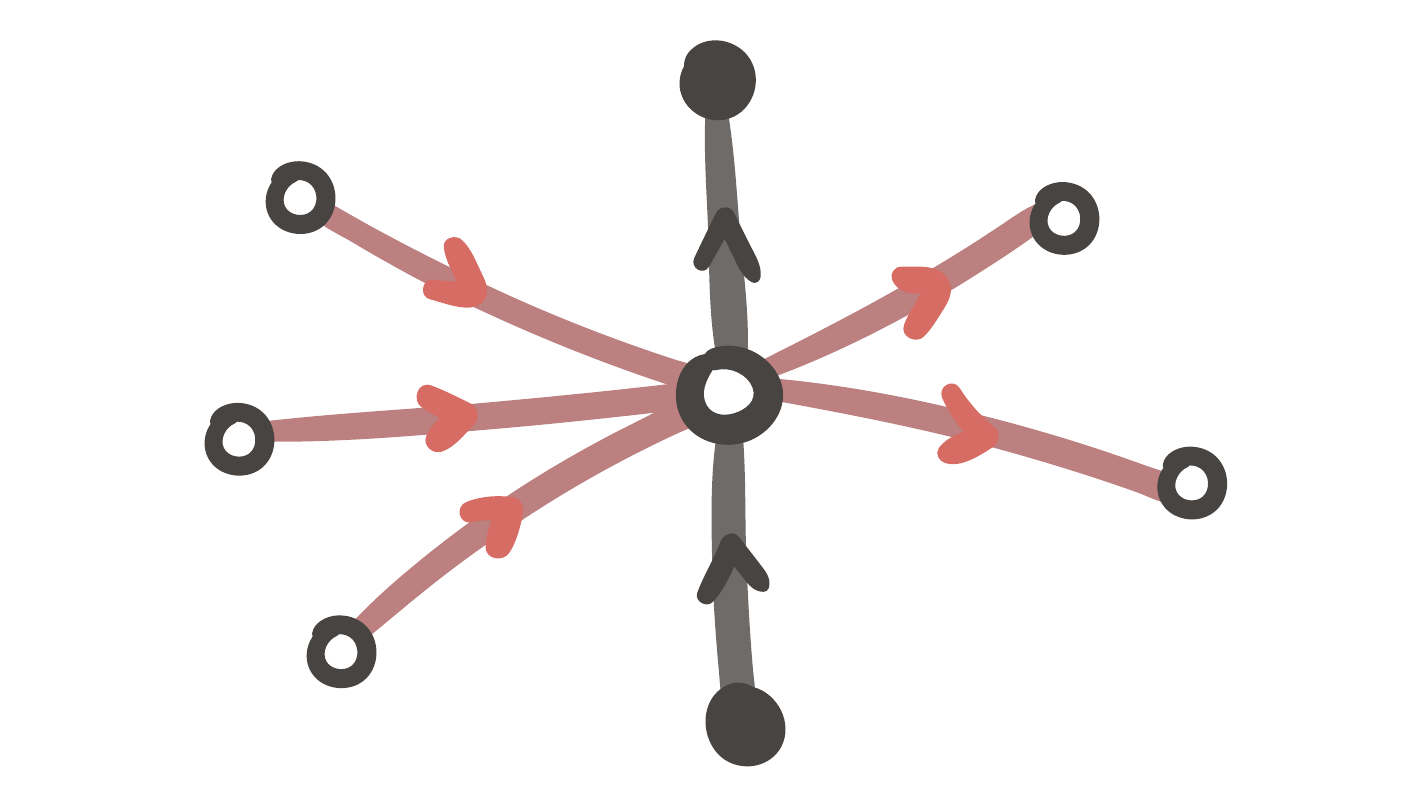}
    		\caption{Around a white vertex $\notin\{S,N\}$.}
    		\label{subfig:white_vertex_V-trans}
     \end{subfigure}
\end{center}
\caption{Configurations around vertices in the intermediate map~$B'$.}
\end{figure}
The other way around, let~$P$ be a plane bipolar poset, with white vertices and red edges,  and let~$B'$ be obtained from $P$ by the following additions :
	\begin{itemize}
	\item We add a black vertex~$v_f$ in every inner face~$f$ of~$P$. We also add two black vertices~$N', S'$ in the outer face of $P$, respectively on the left side and on the right side of $P$. 
	\item For every inner face~$f$ of $P$, 
	we add black edges connecting $v_f$ to all non-extremal vertices around $f$; the edges connected to the left (resp. right) lateral path of $f$ are directed from (resp. toward) $v_f$, as in Figure~\ref{subfig:red_face_V-trans}. 
	\item For every vertex $v\notin\{S,N\}$ on the left (resp. right) boundary of $P$, we add a black edge from $N'$ to $v$ (resp. from $v$ to $S'$).
	\end{itemize}
	At every non pole white vertex, we have the configuration shown in~Figure~\ref{subfig:white_vertex_V-trans}.
	Hence, the part made of the black edges and the vertices $\notin\{S,N\}$ is the 2-subdivision of an oriented map $B$.
	Moreover, by Figure~\ref{subfig:red_face_V-trans}, $B$ has a single source at $S'$ and a single sink at $N'$. We also observe (see Figure~\ref{subfig:white_vertex_V-trans}) that whenever there is an edge from $u$ to $v$
	in $B$,  in $P^*$ (the dual of $P$, in the sense of~\cite{de1995bipolar}) there is a directed path from $u$ to $v$. Since $P^*$ is acyclic, $B$ is thus also acyclic, so that it is a plane bipolar orientation. We let $\psi$ be the mapping that 
	associates $B$ to $P$. 
		
To show that~$\phi$ and~$\psi$ are inverse of one another, we observe that black vertices and black edges removed by the construction~$\phi$ are those added by~$\psi$.  
Moreover, red edges we add with~$\phi$ on the left and on the right of~$B$'s outer face are the edges incident to~$S$ and~$N$ in~$P$, which we remove in the construction~$\psi$.

Applying the KMSW bijection to $V$-transverse bipolar orientations, and using Lemma~\ref{res:transversal_completion}, 
we obtain the bijective correspondence with $V$-admissible tandem walks as stated.  
\end{proof}

\subsection{Application to transversal structures}\label{sec:appli_ts}

Another kind of oriented maps to be related below to weighted quadrant tandem walks are transversal structures~\cite{he1993finding,kant1997regular,fusy2009transversal} (which  encode the combinatorial types of generic rectangulations). 
A \emph{4-outer map} is a simple map whose outer face contour is a (simple) 4-cycle, the outer vertices being denoted $W,N,E,S$ in clockwise order.  A \emph{transversal structure} on such a map (see Figure~\ref{fig:transversal}) is an orientation and  bicoloration of its inner edges (in blue or red) so that the following local conditions are satisfied:  
\begin{itemize}
\item[(T1):] the edges incident to $W,N,E,S$ are respectively outgoing blue, ingoing red, ingoing blue, and outgoing red;
\item[(T2):] for each inner vertex $v$, the incident edges in clockwise order around $v$ form four (non-empty) groups: outgoing red, outgoing blue, ingoing red, ingoing blue. 
\end{itemize}
The \emph{$WE$-type} of a transversal structure is the integer pair $(\mathrm{deg}(W)-2,\mathrm{deg}(E)-2)$. 

\begin{figure}
\begin{center}
\begin{subfigure}[b]{0.3\textwidth}
         \centering
         \includegraphics[width=\textwidth]{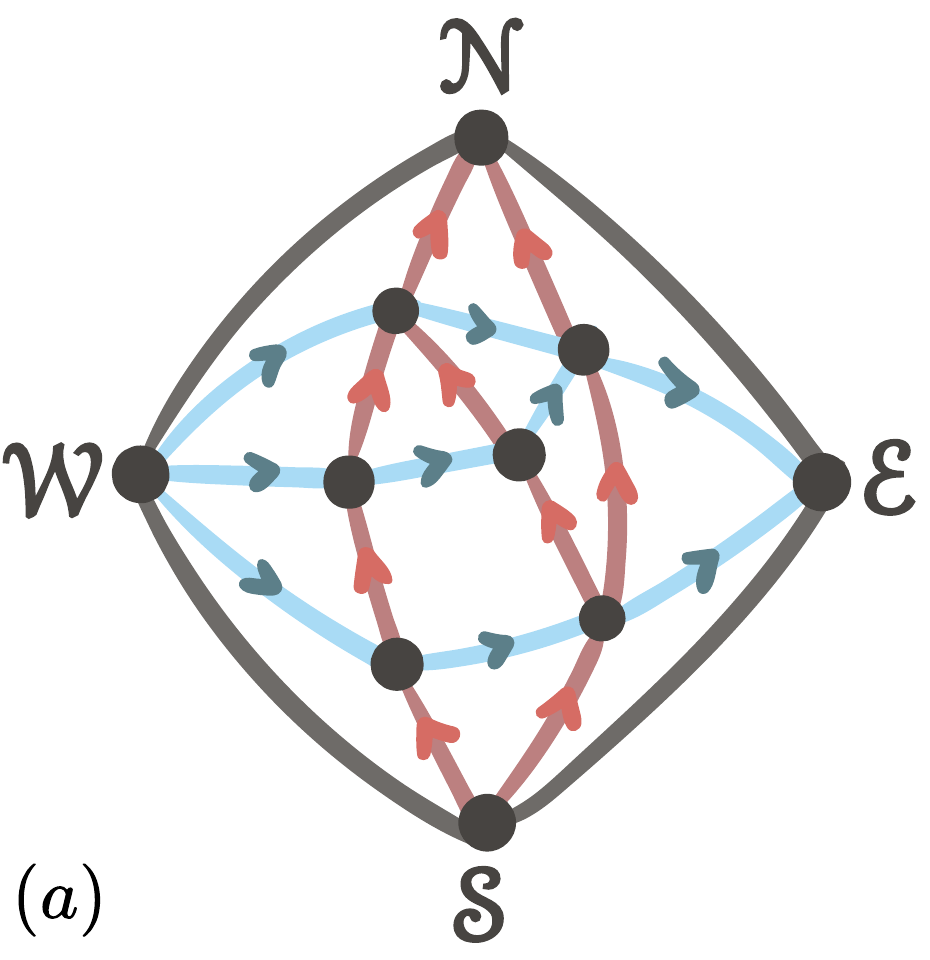}
    
     \end{subfigure}
     \begin{subfigure}[b]{0.3\textwidth}
         \centering
         \includegraphics[width=\textwidth]{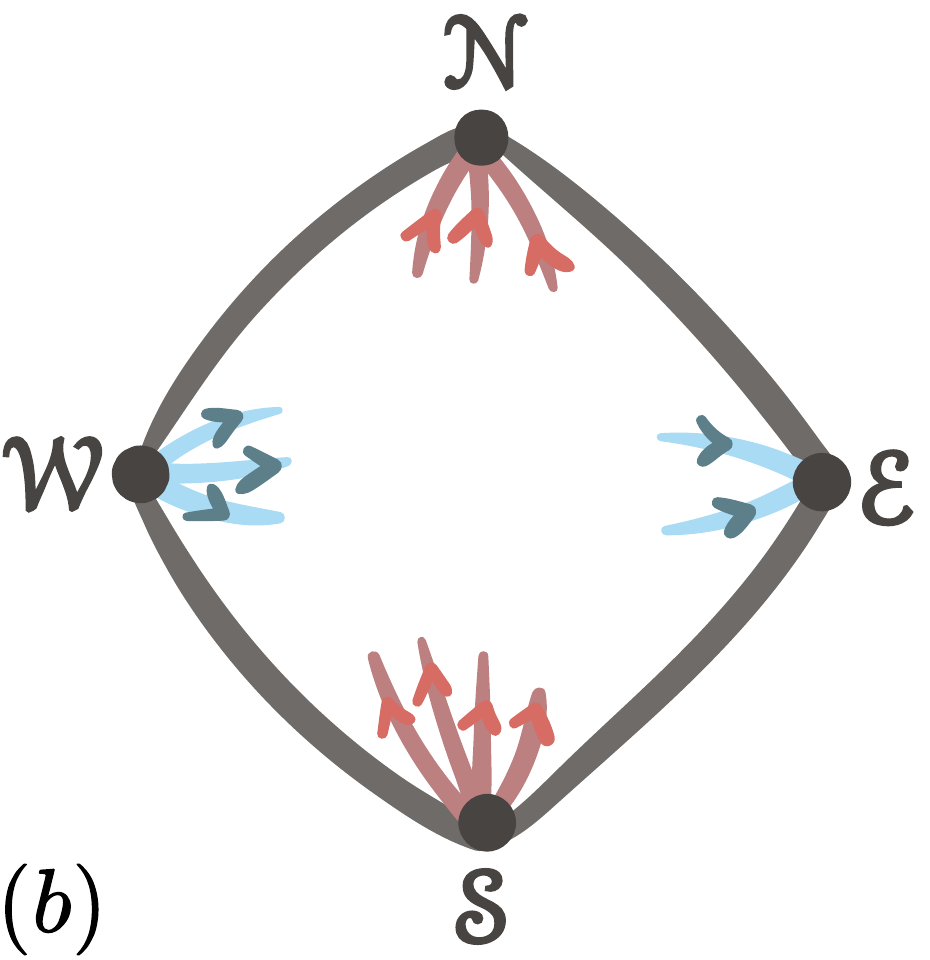}
     \end{subfigure}
          \begin{subfigure}[b]{0.3\textwidth}
         \centering
         \includegraphics[width=\textwidth]{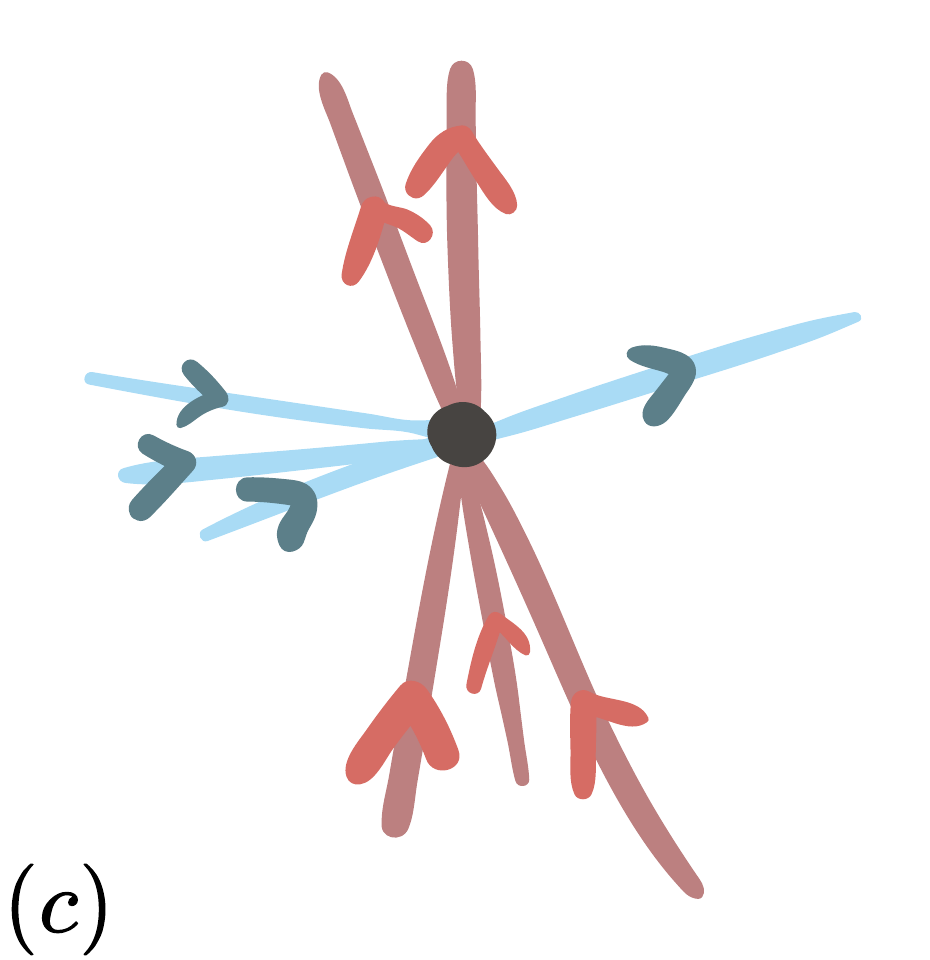}
     \end{subfigure}
\end{center}
\caption{(a) A  4-outer map endowed with a transversal structure. (b) The local condition (T1). (c) The
local condition (T2).}
\label{fig:transversal}
\end{figure}

\begin{lemma}\thlabel{res:T:acyclic_poset}
In a transversal structure $X$, let $X_r$ be the spanning oriented map formed by the red edges, and the $4$ outer edges, recolored red and oriented so as to form 
two directed paths from $S$ to $N$. Then $X_r$ is acyclic and defines a poset. All inner faces have degree in $\{3,4\}$ and must be of one of the five types shown in Figure~\ref{fig:five_lcdt_faces}. Inner faces incident to~$W,N,E,S$ must have degree~$3$.
\end{lemma}
\begin{proof}
Assume that $X_r$ has a directed cycle, and let $\gamma$ be a minimal one, i.e. whose interior~$\gamma^\circ$ does not contain the interior of another directed red cycle of~$X_r$. Condition~(T2) implies that if~$\gamma$ is clockwise (resp. counterclockwise), then any blue edge in~$\gamma^\circ$, incident to a vertex~$v$ on~$\gamma$, must be outgoing (resp. ingoing) at~$v$. Hence, $\gamma$ has no red chord inside, so that there must be at least one vertex~$v_0$ in~$\gamma^\circ$.
	
	From~$v_0$ starts a path of outgoing blue edges (every vertex has at least one outgoing blue edge) that can not loop by minimality of~$\gamma$, hence has to reach a vertex~$v'$ on~$\gamma$. Similarly, from~$v_0$ starts a path of ingoing blue edges that can not loop and reaches a vertex~$v''\neq v'$ on~$\gamma$. These two paths together with the path on~$\gamma$ connecting~$v'$ and~$v''$, form a directed cycle whose interior is in~$\gamma^\circ$, a contradiction. Conditions~(T1) and~(T2) ensure that~$S$
	is the only source and $N$ the only sink of $X_r$. Hence, $X_r$ is a plane bipolar orientation.  
	
	Now we show that $X_r$ is a poset. Assume it has an inner face $f$ whose type has a zero entry. Only one entry is zero, since the underlying map is simple.
	Assume e.g. $f$ has type $(i,0)$ for some $i\geq 1$. Then there is a vertex $v$ in the interior of the left lateral path of $f$. By Condition~(T2), there is at least one
	blue edge $e$ in $f$ starting from $v$. Since there is no vertex in the interior of the right lateral path of $f$, no corner around $f$ can possibly have the ingoing part 
	of $e$ so as to satisfy~(T2). Similarly, there can be no inner face of type $(0,j)$ for some $j\geq 1$. Hence, $X_r$ is a plane bipolar poset. 
	It is similarly easy to check that, within each inner face $f$ 
	of $X_r$, the blue edges have to go from the interior of the left lateral path to the interior of the right lateral path of $f$ (in order to have (T2) satisfied), and 
	all faces within $f$ have to be of the five types shown in Figure~\ref{fig:five_lcdt_faces}.   
	\end{proof}

\begin{figure}
\begin{center}
\includegraphics[width=\textwidth]{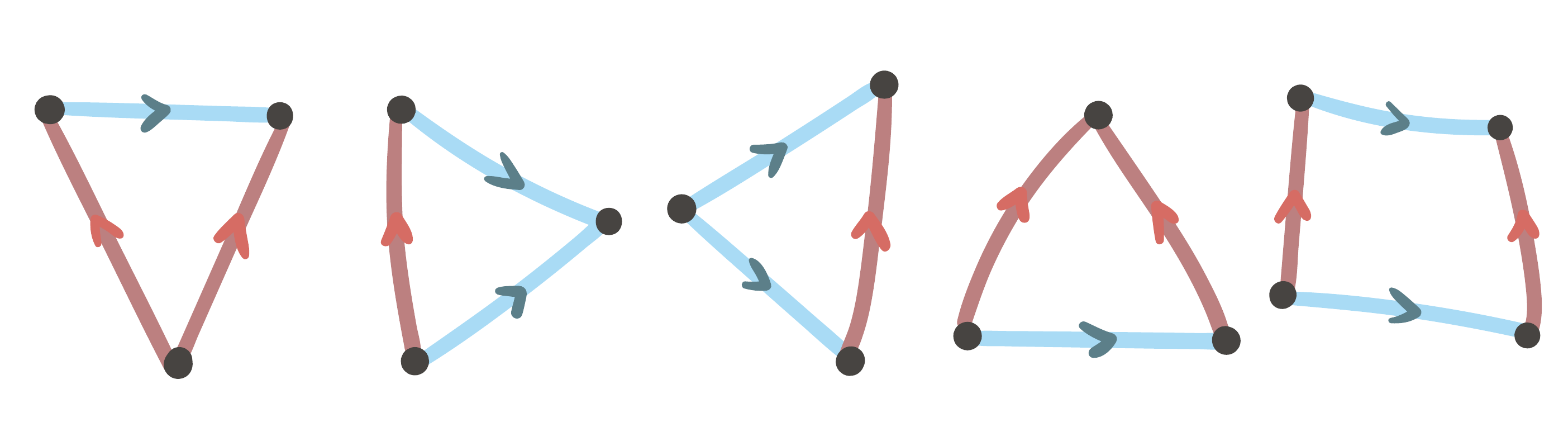}
\end{center}
\caption{The five types of inner faces in a transversal structure. 
}
\label{fig:five_lcdt_faces}
\end{figure}

\begin{remark}
Lemma~\ref{res:T:acyclic_poset} implies that a 4-outer map admitting a transversal structures has to have all its face-degrees in $\{3,4\}$. In the 
reverse direction, it is known that if a 4-outer map with triangular inner faces has no non-facial triangle, then it admits a transversal structure~\cite{he1993finding}. 
Such transversal structures are called \emph{triangulated}. 
As is also known~\cite{he1993finding}, transversal structures naturally arise as the dual to rectangular tilings, the quadrangular faces corresponding to ``degenerate" vertices in the tiling, where $4$ rectangles meet. This correspondence is illustrated in Figures~\ref{fig:rectangular} and~\ref{fig:rectangular_quad}. \hfill$\triangle$
\end{remark}

\begin{figure}
\begin{center}
\includegraphics[width=\textwidth]{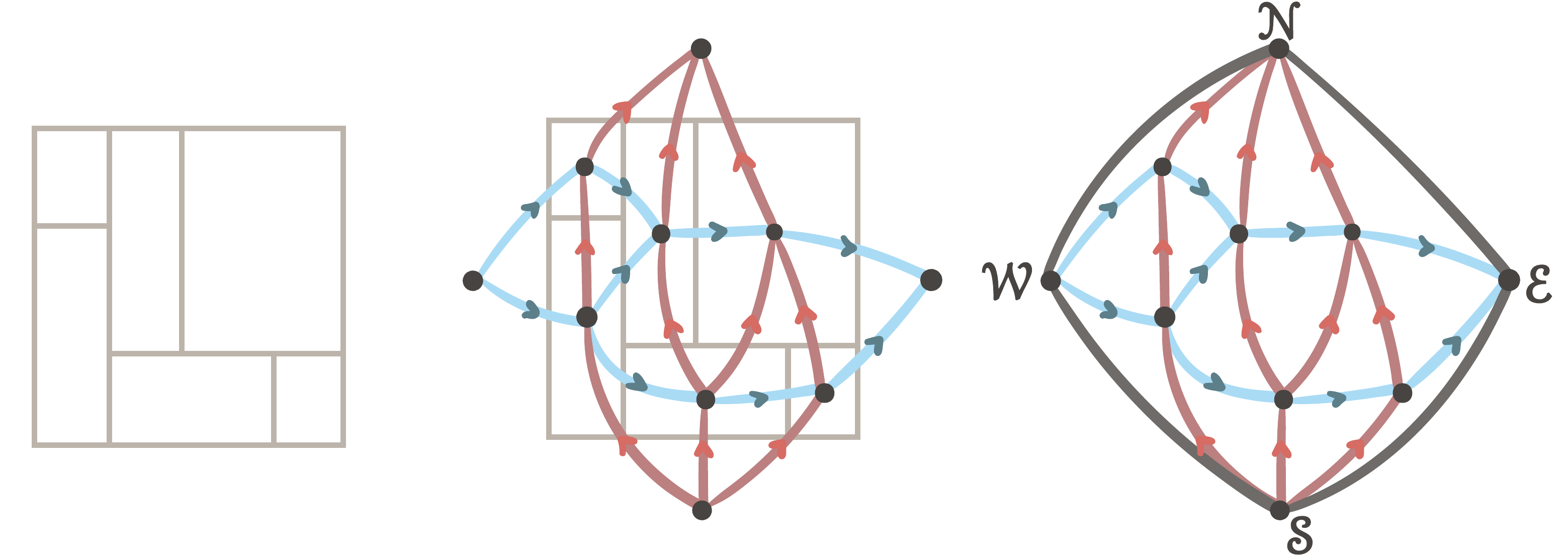}
\end{center}
\caption{A ``generic" rectangular tiling, and its corresponding triangulated transversal structure.}
\label{fig:rectangular}
\end{figure}
\begin{figure}
\begin{center}
\includegraphics[width=\textwidth]{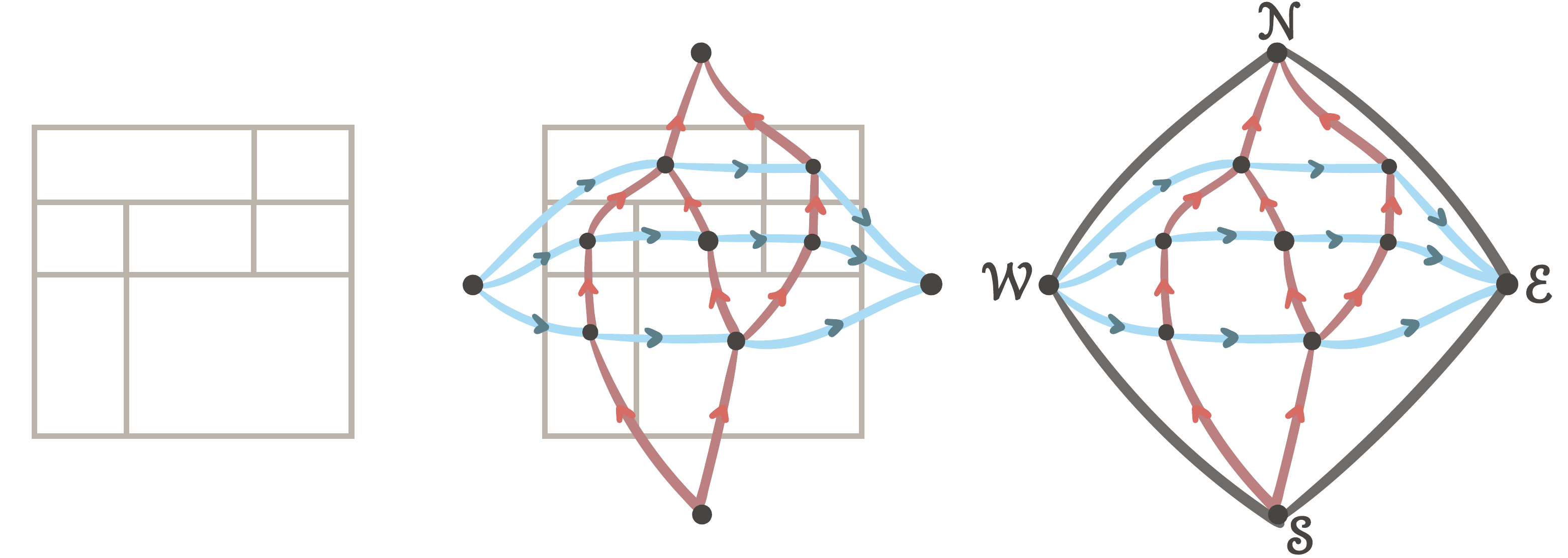}
\end{center}
\caption{A rectangular tiling with two ``degenerate" vertices where 4 rectangles meet, and its corresponding transversal structure.}
\label{fig:rectangular_quad}
\end{figure}
As we are now going to show, transversal structures are bijectively related to certain decorated plane bipolar posets 
  which we call T-transverse bipolar orientations (in doing this, we break the symmetry of the roles played by red edges and blue edges in transversal structures).  
Precisely, for a plane bipolar poset (with red edges), a \emph{transversal addition} consists in the planar addition of so-called \emph{transversal edges} (blue edges) in each inner face~$f$, each such edge directed from a vertex in the interior of the left lateral path to a vertex in the interior of the right lateral path of~$f$, such that after addition of these edges, all faces within~$f$ are of the types shown in~Figure~\ref{fig:five_lcdt_faces}. A~\emph{T-transverse bipolar orientation} is defined as a plane bipolar poset endowed with a transversal addition, see the left-part of Figure~\ref{fig:t-transverse_example} for an example.    

\begin{figure}
\begin{center}
\includegraphics[width=0.7\textwidth]{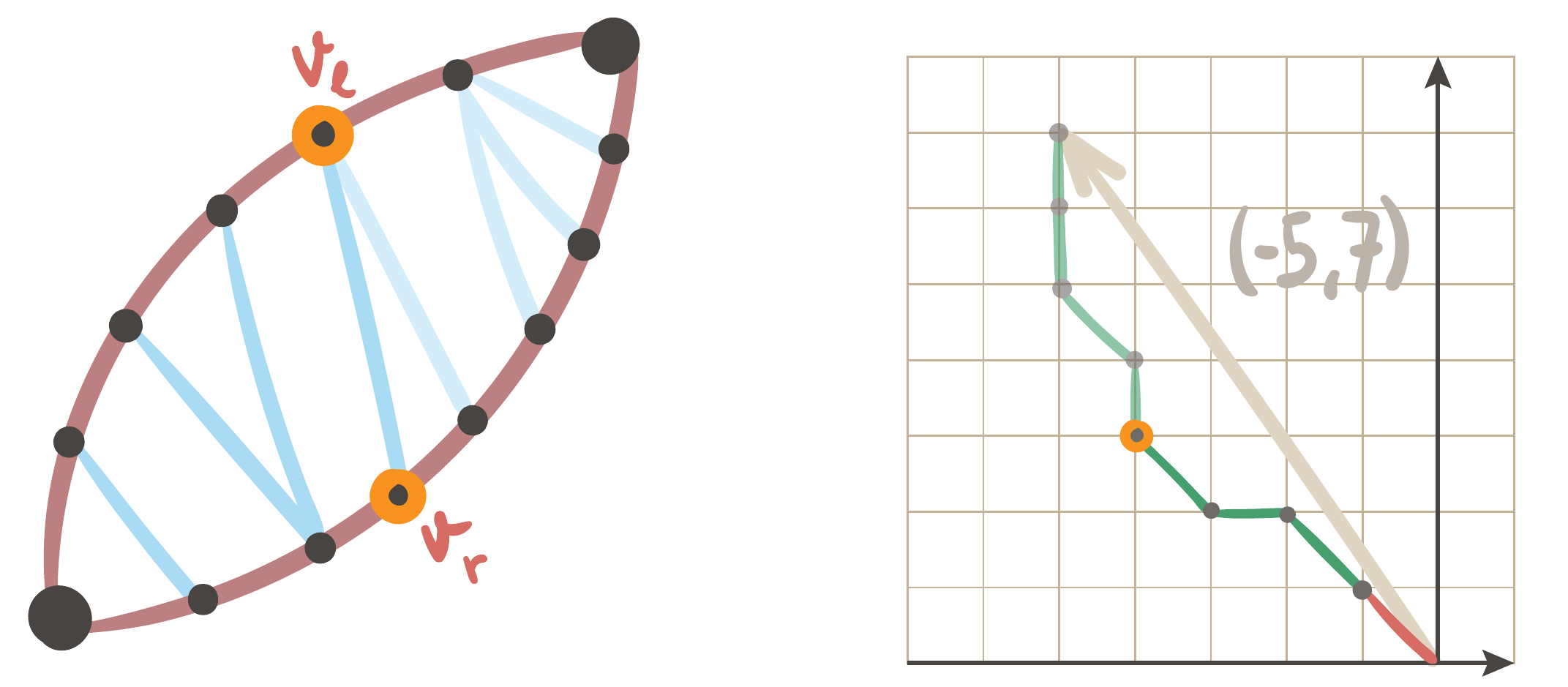}
\end{center}
\caption{The configuration of transversal edges in an inner face of type $(i,j)$ (here $i=5,j=7$) of the red bipolar poset  of a $T$-transverse bipolar orientation, 
 and the small step walk attached to the corresponding face-step~$(-i,j)$.}
\label{fig:kmsw_proof_R3}
\end{figure}

\begin{lemma}\thlabel{res:transversal_addition_lcdt}
	Let~$B$ be a plane bipolar poset, and let~$f$ be an inner face of type~$(i,j)$ in $B$.  
	Then the admissible configurations for transversal edges within $f$ ---in transversal additions on $B$--- 
	are encoded by  walks from~$(0,0)$ to~$(-i,j)$ with steps in~$\{W,N,NW\}$ and starting
	with a step $NW$.
\end{lemma}
\begin{proof}
Given a walk $\gamma$ from $(0,0)$ to~$(-i,j)$ with steps in~$\{W,N,NW\}$ and with $NW$ as first step, the transversal edges within $f$ are inserted from bottom to top 
while reading $\gamma$. Precisely, we maintain two marked vertices:~$v_\ell$ on the left lateral path of~$f$ and~$v_r$ on its right lateral path, initially~$v_\ell$ and~$v_r$ are both at the bottom vertex of~$f$. When reading
a step $(-\varepsilon_\ell, \varepsilon_r)\in\{W,N,NW\}$ in $\gamma$, we move up~$v_\ell$ (resp.~$v_r$) by~$\varepsilon_\ell$ edge (resp. by~$\varepsilon_r$ edge) on the left (resp. right) lateral path of~$f$; then we add an oriented transversal edge from~$v_\ell$ toward~$v_r$. Since $\gamma$ starts with $NW$, the first added transversal
edge forms with the bottom vertex of $f$ a face of the 1st type shown in Figure~\ref{fig:five_lcdt_faces}. In addition, every face enclosed between two consecutive transversal edges
has to be of the 2nd type (for each step $W$), or the 3rd type (for each step $N$), or the 5th type (for each step $NW$). Since $\gamma$ ends at $(-i,j)$ 
and $f$ has type $(i,j)$, the last added transversal edge forms with the top-vertex of $f$ a face of the 4th type. 

Conversely, every admissible configuration within $f$ clearly yields such a walk, 
by writing an initial step $NW$, then traversing the faces in $f$ from bottom to top and writing a step $W$ (resp. $N,NW$) each time a face of the 2nd (resp. 3rd, 5th) type is traversed.  
\end{proof}

A \emph{$T$-admissible tandem walk} is a tandem walk where each face-step $(-i,j)$ satisfies $i\geq 1$ and $j\geq 1$, 
and to each such step is attached a walk with steps in $\{W,N,NW\}$ starting with a step $NW$ considered as \emph{marked}, and with same starting and ending point as the face-step.

\begin{proposition}\label{prop:transversal}
Transversal structures of $WE$-type $(p,q)$, having $n$ inner vertices and $m$ red edges, are in bijection with $T$-transverse bipolar orientations of outer type $(p,q)$,
 having $m$ plain edges and $n+2$ vertices. These are in bijection to $T$-admissible tandem walks of length $m-1$ from $(0,p)$ to $(q,0)$, with $n$ SE steps. 
 Each quadrangular inner face in the transversal structure corresponds to an unmarked NW step in an attached walk of a face-step.  
\end{proposition}
\begin{proof}
Starting from $M$ a $T$-transverse bipolar orientation of outer type $(p,q)$, with $B$ the underlying plane bipolar poset, 
we add vertices $W,E$ in the outer face of $B$, respectively on the left side and on the right side,
and for every vertex $v\notin\{S,N\}$ on the left (resp. right) boundary of $B$, we add a blue edge from $W$ to $v$ (resp. from $v$ to $E$). 
We then add  the edges $(W,N),(N,E), (E,S), (S,W)$ (unoriented, uncolored) to form an outer quadrangle. Let $X$ be the obtained structure. 
Clearly, Condition~(T1) is satisfied. In addition, the local condition at non-pole vertices of plane bipolar orientations ensures that 
each inner vertex $v$ of $X$ has a non-empty group of ingoing red edges, and a non-empty group of outgoing red edges. Let $f_{\ell}$
be the face of $B$ (possibly the ``left outer face'') incident to the left lateral corner of $v$, and let $f_r$ be the face of $B$ (possibly the ``right outer face")
incident to the right lateral corner of $v$. Since all inner faces of $X$ (upon recoloring red the outer edges and directing them as two paths from $S$ to $N$)
are of the type shown in Figure~\ref{fig:five_lcdt_faces}, all ingoing (resp. outgoing)  blue edges incident to $v$ have to be in $f_{\ell}$ (resp. $f_r$), and there 
has to be at least one such edge. Hence, Condition~(T2) is satisfied, so that $X$ is a transversal structure, of WE-type $(p,q)$. 
We let $\phi$ be the mapping that associates $X$ to $M$.  

Conversely, from $X$ a transversal structure of WE-type $(p,q)$, we simply remove $W,E$ and their incident edges. By Lemma~\ref{res:T:acyclic_poset}, we obtain a $T$-transverse bipolar orientation $M$. Let $\psi$ be the mapping that associates $M$ to $X$. Clearly, the two mappings are inverse of one another, hence give a bijection. We give in Figure~\ref{fig:t-transverse_example} an illustration of this bijection.

\begin{figure}
\begin{center}
\includegraphics[width=0.8\textwidth]{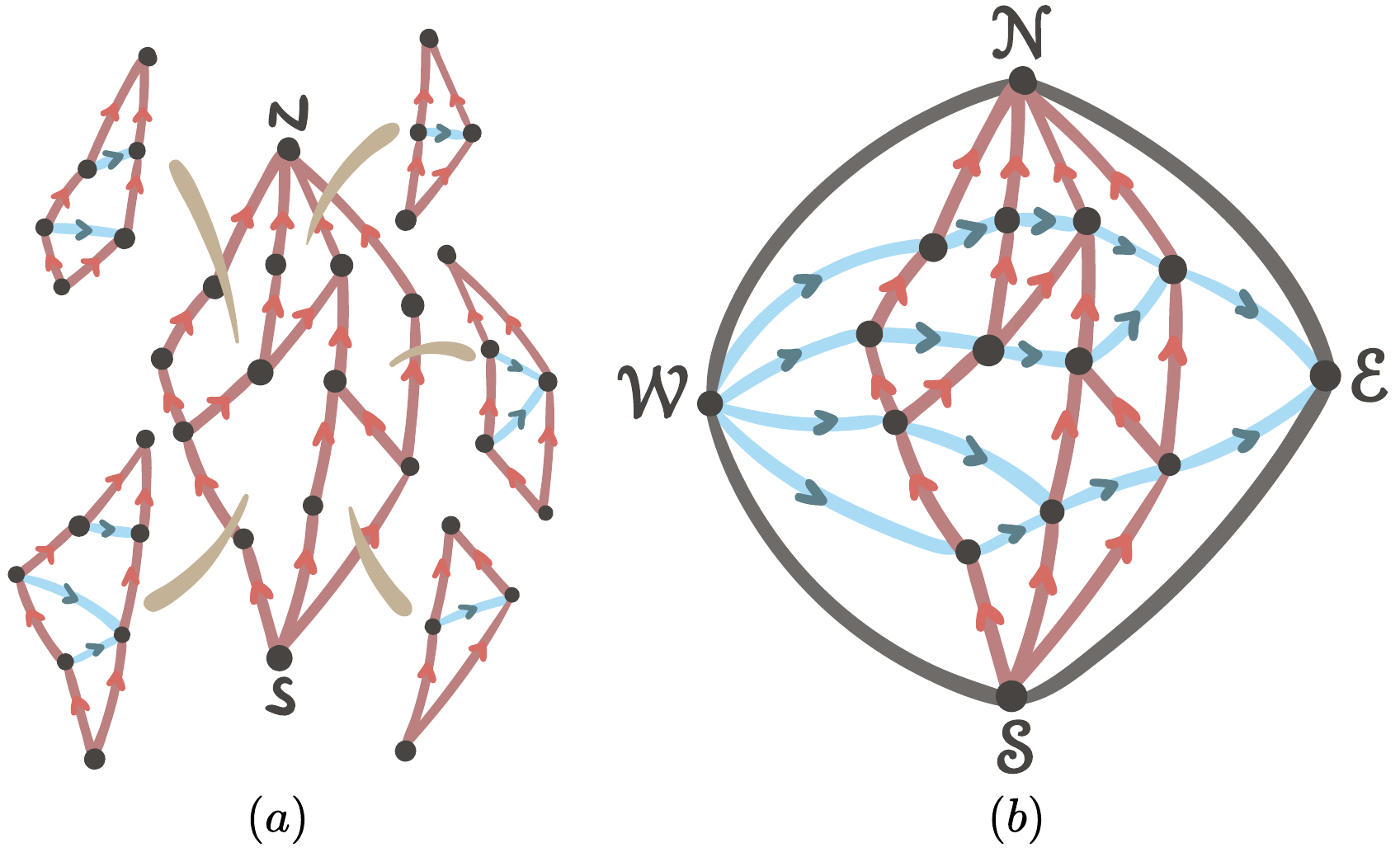}
\end{center}
\caption{(a) A $T$-transverse bipolar orientation. (b) The corresponding transversal structure.}
\label{fig:t-transverse_example}
\end{figure}  

 Applying the KMSW bijection to $T$-transverse bipolar orientations, and using Lemma~\ref{res:transversal_addition_lcdt}, 
we obtain the bijective correspondence with $T$-admissible tandem walks as stated. 
\end{proof}

\section{Exact counting results}\label{sec:exac}
In this section, we obtain exact enumeration results for plane bipolar posets and transversal structures. 
After rephrasing the results of the previous section in the general terminology of counting weighted tandem walks in the quadrant (Theorem~\ref{prop:weight}),  
we give a unified functional equation, which we call the master equation, for the generating function of weighted tandem walks in the quadrant. 
Manipulations on the master equation then imply that plane bipolar posets counted by edges are 
equinumerous to quadrant excursions for an explicit set of small steps (Proposition~\ref{prop:small_step_plane_posets_edges}), and that 
 plane bipolar posets counted by vertices are equinumerous to plane permutations (Proposition~\ref{prop:plane_perm}), which have been counted in~\cite{bouvel2018semi}. 
On the other hand, for transversal structures (with 
a weight-parameter for quadrangular inner faces), we resort to a rewriting of the corresponding tandem walks into a model of quadrant walks with small steps (and 
some forbidden patterns for consecutive steps), which yields an explicit recurrence for the counting coefficients (Proposition~\ref{prop:rec_transversal}).

For $w(i,j)$ a weight-function (with $(i,j)\in\mathbb{N}^2$), 
a \emph{$w$-weighted tandem walk} is a tandem walk where every face-step $(-i,j)$ has weight $w(i,j)$. 
The \emph{$w$-weight} of such a walk is the product of the weights of its face-steps; and for a given finite set of  walks, the associated \emph{$w$-weighted number} is 
the sum of $w$-weights of the walks in the set. 
Via the results obtained in the previous section, we can reformulate the enumeration
of plane bipolar posets and of transversal structures in terms of weighted enumeration of quadrant tandem walks. 
In the statement below, transversal structures are counted with weight $v$ per quadrangular inner face; accordingly we use the terminology of \emph{$v$-weighted number} of transversal structures. The case $v=0$ (resp. $v=1$) gives the enumeration of triangulated (resp. general) transversal structures.

\begin{theorem}\label{prop:weight}
Let $w(i,j)=\mathbf{1}_{i>0,j>0}$. Then the number of plane bipolar posets with $n$ edges and outer type $(a,b)$ 
is equal to the $w$-weighted number of quadrant tandem walks of length $n-1$ from $(0,a)$ to $(b,0)$.   

Let $w(i,j)=\binom{i+j}{i}$. Then the number of plane bipolar posets with $n+2$ vertices and pole-type $(p,q)$ is equal to the  
$w$-weighted number of quadrant tandem walks of length $n-1$ from $(0,p)$ to $(q,0)$. 

Let $w(i,j)=\sum_{r\geq 0}\frac{(i+j-2-r)!}{(i-1-r)!(j-1-r)!r!}v^r$ (so $w(i,j)=\binom{i+j-2}{i-1}$ for $v=0$).  
Then the $v$-weighted number of transversal structures of WE-type $(p,q)$, with $m$ red edges and $n+4$ vertices,  is equal to the $w$-weighted number of quadrant tandem walks
 of length $m-1$ from $(0,p)$ to $(q,0)$, having $n$ SE steps.    
\end{theorem}
\begin{proof}
This is a direct consequence of Proposition~\ref{prop:E-admissible}, Proposition~\ref{prop:posets_vertices}, and Proposition~\ref{prop:transversal}, respectively. 
Indeed, for plane bipolar posets counted by edges, the weight just filters those tandem walks with no face-step having a zero entry; while 
for plane bipolar posets counted by vertices, and for transversal structures,  the 
weight $w(i,j)$ corresponds to the number of ways to choose the walk attached to a face-step $(-i,j)$, as shown in Figure~\ref{fig:trans_completion} and Figure~\ref{fig:kmsw_proof_R3}.  
\end{proof}

Let $P^w_a(x,y)$ denote the generating series of $w$-weighted quadrant tandem walks starting in position $(0,a)$, with respect to the number of steps (variable $t$), end positions (variables $x$ and $y$) and number of SE steps (variable $u$). A last step decomposition immediately yields the following \emph{master equation} 
in the ring 
of formal power series in $t$ and $y$ with coefficients that are Laurent series in $\bar x=1/x$ (and polynomial in $u$): \vspace{-1em}
\begin{align*}
    P^w_a(x,y)&=y^a\!+ t u \frac x y \left(P^w_a(x,y)\!-\!P^w_a(x,0)\right)\!+\!t\sum_{i,j\geq0}\!w(i,j)\frac{y^j}{x^i}\left(\!\!P^w_a(x,y)\!-\!\sum_{k=0}^{i-1}x^{k}[x^k]P^w_a(x,y)\!\right)\\
    &=y^a\!+\!t u \frac x y \left(P^w_a(x,y)\!-\!P^w_a(x,0)\right)\!
      +\!tW_{0}(\bar x,y)P^w_a(x,y)\!
      -\!t\sum_{k\geq0}W_{k+1}(\bar x,y)x^{k}[x^k]P^w_a(x,y)\\[-2.5em]
\end{align*}
where $W_{k}(\bar x,y)=\sum_{i\geq k,j\geq0}w(i,j)\frac{y^j}{x^i}$. The subtracted terms correspond to the cases of adding a step that makes the walk
leave the quadrant. For the addition of a face-step $(-i,j)$, note that the walk leaves the quadrant iff it ends at some abscissa $k\in\{0,\ldots,i-1\}$.



\subsection{Plane bipolar posets by edges} The case of bipolar posets counted by edges  corresponds to having $w(i,j)=\mathbf{1}_{i\neq 0,j\neq 0}$ (cf Claim~\ref{claim:tra}). The master equation then becomes
\[
E_a(x,y)=y^a+t\frac{x}{y}\big(E_a(x,y)-E_a(x,0)\big)+t\frac{\bar{x}}{1-\bar{x}}\frac{y}{1-y}\big(E_a(x,y)-E_a(1,y)\big),
\]
and the coefficient $[t^n x^b]E_a(x,0)$ gives the number of plane bipolar posets of outer type $(a,b)$ with $n+1$ edges. 

By some algebraic manipulations similar to those performed in~\cite[Sect.5.2]{bouvel2018semi}, we can relate the enumeration of plane bipolar posets by edges to a simpler model of quadrant walks:
\begin{proposition}\label{prop:small_step_plane_posets_edges}
For $n\geq 1$, let $e_n$ be the number of plane bipolar posets with $n$ edges. Then $e_n$ is equal to the number of quadrant excursions of length $n-1$ with steps in $\{0,E,S,NW,SE\}$.
\end{proposition}
\begin{proof}
Note that, for $n\geq 2$, $e_n$ is also (by adding a path of length $2$ from source to sink on the left side) the number of plane bipolar posets with $n+2$ edges
and left boundary of length $2$, thus $e_n=[t^{n+1}]E_1(1,0)$. 
Let $Q(u,v):=E_1(x,y)$ under the change of variable relation $\{x=1+u, \bar{y}=1+\bar{v}\}$ (note that $E_1(1,0)=Q(0,0)$). 
Via the change of variable, the functional equation for $E_1(x,y)$ becomes
\[
Q(u,v)=1+t(1+u+\bar{v}+u\bar{v}+\bar{u}v)Q(u,v)-t(1+u)(1+\bar{v})Q(u,0)-t\bar{u}vQ(0,v).
\]
This resembles the functional equation for the series $\hQ(u,v)$ of quadrant walks (starting at the origin) with steps in $\{0,E,S,NW,SE\}$, whose functional equation is 
\[
\hQ(u,v)=1+t(1+u+\bar{v}+u\bar{v}+\bar{u}v)\hQ(u,v)-t(1+u)\bar{v}\hQ(u,0)-t\bar{u}v\hQ(0,v).
\]
By coefficient extraction in each of these two functional equations, we recognize that $Q(u,v)$ and $\hQ(u,v)$ appear to be related as  $Q(u,v)=\frac{v}{1+v}+t(1+u)-t^2+t^2(1+v)\hQ(u,v)$. This relation can then be easily checked. Indeed, if we substitute $Q(u,v)$ 
by $\frac{v}{1+v}+t(1+u)-t^2+t^2(1+v)\hQ(u,v)$ in the first functional equation, we recover the second functional equation (multiplied by $t^2(1+y)$).   
As a consequence $[t^n]Q(0,0)=[t^{n-2}]\hQ(0,0)$ for $n\geq 3$. Thus for $n\geq 2$, $e_n= [t^{n+1}]E_1(1,0)=[t^{n+1}]Q(0,0)=[t^{n-1}]\hQ(0,0)$
(and for $n=1$ one manually checks that $e_1=1=[t^0]\hQ(0,0)$). 
\end{proof}

While the series $E_1(1,0)$ is non D-finite as discussed in the next section, the reduction to a quadrant walk model with small steps ensures that the sequence $e_1,\ldots,e_n$ can be computed with time complexity $O(n^4)$ and using $O(n^3)$ bit space. The sequence starts as $1,1,1,2,5,12,32,93,279,872,2830,\ldots$, it is A363682 in~\cite{oeis}.


\subsection{Plane bipolar posets by vertices}
In the case of plane bipolar posets enumerated by vertices, by Proposition~\ref{prop:weight} we have $w(i,j)={i+j\choose i}$ for $i,j\geq0$, so that $W_k(\bar x,y)=\frac{1}{1-(\bar x+y)}\frac{\bar x^{k}}{(1-y)^{k}}$ in $\mathbb{Q}[[y,\bar x]]$. The master equation 
then rewrites
\begin{align*}
  B_a(x,y)
      &\!=\!y^a\!+ t \frac x y \left(B_a(x,y)\!-\!B_a(x,0)\right)\!
      +\!\frac{t}{1\!-\!y}\ \!\frac1{x\!-\!\frac{1}{1-y}}
      \left(xB_a(x,y)-\frac{1}{1\!-\!y}B_a\left(\frac1{1\!-\!y},y\right)\right),
      \label{master2}
\end{align*}
and the coefficient $[t^n x^b]B_a(x,0)$ gives the number of plane bipolar posets of pole-type $(a,b)$ with $n+3$ vertices.

From the functional equation, we now prove that plane bipolar posets counted by vertices are equinumerous to so-called \emph{plane permutations} introduced in~\cite{bousquet2007forest} and that have been recently studied in~\cite{bouvel2018semi}. These are the permutations avoiding the vincular pattern $2\underbracket[.5pt][1pt]{14}3$, 
i.e., with no pattern $2143$ where $1$ and $4$ are adjacent. Plane permutations also the permutations such that the dominance poset of the point diagram is planar.  

\begin{proposition}\label{prop:plane_perm}
Let $b_n$ denote the number of plane bipolar posets with $n+2$ vertices. Then $b_n$ is equal to the number of plane permutations on $n$ elements.
\end{proposition}
\begin{proof}
  Note that $b_n$ is also, by adding a new source of degree $1$
  (connected to the former source), the number of plane bipolar posets
  of pole-type $(0,j)$ with $n+3$ vertices and arbitrary $j\geq 0$, so
  that $b_n=[t^{n}]B_0(1,0)$. Let $S(u,v):=v(B_0(v,1-\bar u)-1)$, so
  that, under the change of variable $\{y=1-\bar u, x=v\}$, we have
  $S(u,v)=x(B_0(x,y)-1)$, $B_0(1,0)=1+S(1,1)$ and
  $\frac1{1-y}B_0\left(\frac1{1-y},y\right)=uB_0(u,1-\bar
  u)=S(u,u)+u$. The main equation for $B_0$ then rewrites
\begin{align*}
  S(u,v)&=x(B_0(x,y)-1)\\
      &\!=\!x(y^0-1)\!+ t \frac x y \left(xB_0(x,y)\!-\!xB_0(x,0)\right)\!
      +\!\frac{tx}{1\!-\!y}\ \!\frac1{x\!-\!\frac{1}{1-y}}
      \left(xB_0(x,y)-\frac{1}{1\!-\!y}B_0\left(\frac1{1\!-\!y},y\right)\right)\\
      &=\frac{tuv}{u-1}\left(S(u,v)+v-S(1,v)-v\right)
      +\frac{tuv}{v-u}\left(S(u,v)+v-S(u,u)-u\right)\\
      &=tuv+\frac{tuv}{1-u}\left(S(1,v)-S(u,v)\right)
      +\frac{tuv}{v-u}\left(S(u,v)-S(u,u)\right).
\end{align*}
This equation for $S(u,v)$ is  exactly \cite[Eq.~(2)]{bouvel2018semi} (they use $(x,y,z)$ for our $(t,u,v)$), where $[t^n]S(1,1)$ gives 
the number of plane permutations of size $n$. This concludes the proof, since  $b_n=[t^{n}]B_0(1,0)$ and $B_0(1,0)=1+S(1,1)$.
\end{proof}

\begin{remark}
It is shown in~\cite[Prop~13]{bouvel2018semi} that the generating function of plane permutations is D-finite, and in~\cite[Theo~14]{bouvel2018semi} that the number of plane permutations of size $n$ admits single sum expressions similar to the sum expression for the Baxter numbers. \hfill$\triangle$
\end{remark}

In Section~\ref{sec:bijec} we will give a direct bijective proof of Proposition~\ref{prop:plane_perm}, via a similar approach as in the bijection between Baxter permutations and plane bipolar orientations introduced in~\cite{bonichon2010baxter}.

\subsection{Transversal structures}
Finally, in the case of transversal structures, the weight-function is 
$w(i,j)=\sum_{r\geq 0}\frac{(i+j-2-r)!}{(i-1-r)!(j-1-r)!r!}v^r=[\bar x^iy^j]\frac{\bar x y}{1-\bar x -v\bar x y -y}$, 
 so that $W_0(\bar x,y)=W_1(\bar x,y)=\frac{\bar x y}{1-\bar x -v\bar x y-y}$, and more generally $W_{k+1}(\bar x,y)=\frac{y}{1-y}\frac{1}{x-\frac{1+vy}{1-y}}\left(\frac{\bar x(1+vy)}{1-y}\right)^k$. 
The master equation 
then rewrites
\begin{align*}
  T_a(x,y)
      &=y^a+ t u \frac x y \left(T_a(x,y)-T_a(x,0)\right)
      +\frac{ty}{1-y}\frac1{x-\frac{1+vy}{1-y}}
      \left(T_a(x,y)-T_a\left(\frac{1+vy}{1-y},y\right)\right),
\end{align*}
and $[u^{n}t^{m}]T_1(1,0)$ gives the $v$-weighted number of transversal structures with $n+4$ vertices and $m+1$ red edges, having WE-type $(1,j)$ for $j\geq 1$. 
Upon deleting the outer path $(S,W,N)$, this is also the $v$-weighted number of transversal structures (of arbitrary WE-type)  with $n+3$ vertices and $m-1$ red edges.  
As shown in the next section this series (in the variable $u$, with $t=1$) is non D-finite. However, similarly as for plane bipolar posets counted by edges, we can make the coefficient computation faster by reduction to a model of quadrant walks with small steps (however this time with some forbidden two-step sequences). 

\begin{proposition}\label{prop:rec_transversal}
Let $t_n(v)$ be the $v$-weighted number of transversal structures on $n+4$ vertices. 
Let $\tse_n(i,j), \tnw_n(i,j)$ be the coefficients given by the recurrence, valid
for $n\geq 1$:
\begin{equation}\label{eq:rec_trans}
  \left\{\begin{array}{l}
  \tse_n(i,j)=\tse_{n-1}(i-1,j+1)+\tnw_{n-1}(i-1,j+1),\\
  \tnw_n(i,j)=\tse_{n}(i+1,j-1)+(1+v)\,\tnw_{n}(i+1,j-1)+\tnw_{n}(i+1,j)+\tnw_{n}(i,j-1),
  \end{array}\right.
\end{equation} 
with boundary conditions $\tse_n(i,j)=\tnw_n(i,j)=0$ 
for any $(n,i,j)$ with $n\leq 0$ or $i<0$ or $j<0$ or $i>n$, with the exception (initial condition) of 
$\tse_0(0,1)=1$. 

Then, for $n\geq 1$, we have $t_n(v)=\tse_{n+2}(1,0)$. 
\end{proposition}
\begin{proof}
Upon adjoining the paths $SWN$ and $SEN$ to the red poset, $t_n(v)$ is the $v$-weighted number of transversal structures on $n+6$ vertices, with WE-type $(1,1)$. 
Let $\mathcal{N}_n(i,j)$ be the set of quadrant walks starting at $(0,1)$, ending at $(i,j)$, with steps in $\{SE,W,N,NW,{\bf NW}\}$ (with $\bf{NW}$ for the marked NW steps), 
having $n$ SE step, and such that no step in $\{N,W,NW\}$ comes after a $SE$ step. Such walks are counted with weight $v$ per step $NW$. 
By Proposition~\ref{prop:transversal}, $t_n(v)$ is the $v$-weighted number of walks in $\mathcal{N}_{n+2}(1,0)$. 
For $n\geq 1$ and $i,j\geq 0$, let $\tse_n(i,j)$ (resp. $\tnw_n(i,j)$) be the $v$-weighted number of walks in $\mathcal{N}_n(i,j)$ ending (resp. not ending) with a SE step. 
Then a last step decomposition ensures that $\tse_n(i,j),\tnw_n(i,j)$ satisfy the recurrence above. Note also that a quadrant walk in $\mathcal{N}_n(1,0)$
has to end with a SE step, hence the $v$-weighted number of these walks is $\tse_n(1,0)$. 
\end{proof}

 
Similarly as for plane bipolar posets counted by edges, for any fixed integer $v$,   
the recurrence makes it possible to compute the sequence $t_1(v),\ldots,t_n(v)$ with $O(n^4)$ bit operations using $O(n^3)$ bit space. This 
includes triangulated transversal structures (case $v=0$) and general transversal structures (case $v=1$). 
For the triangulated case, we obtain the same complexity order as the recurrence in~\cite{inoue2009counting}; 
another counting method is described in~\cite{reading2012generic} exploiting a bijection to a certain family of pattern-avoiding permutations. 
On the other hand, a counting method working for general $v$ is given in~\cite{conant2014number} using a growth process and inclusion-exclusion.   
The sequences for $v=0$ and $v=1$ appear as A342141 and A181594 in~\cite{oeis}, they start respectively as 
\[
1,2,6,24,116,642,3938,26194, 186042,\ldots\ \ \ \ \ \ \ 1, 2, 6, 25, 128, 758, 5014, 36194,280433,\ldots
\]
For general transversal structures with control on the number of quadrangular inner faces, the complexity to compute $t_1(v),\ldots,t_n(v)$ (this time seen as polynomials   
in $v$) is now $O(n^5)$ bit operations using $O(n^4)$ bit space. The sequence starts as
\[
1, 2, 6, 24 + v, 116 + 12v, 642 + 114v + 2v^2,  3938 + 1028v + 48v^2,  26194+9220v+ 770v^2+10v^3\ldots 
\]
The inclusion-exclusion method of~\cite{conant2014number} gives another polynomial-time procedure to compute these coefficients.
On the other hand, efficient encoding procedures for general transversal structures have been given in~\cite{saito2012two,takahashi2014compact}, and it should be possible to turn them into a recurrence for $t_n(v)$ (extending the one in~\cite{inoue2009counting}) with same complexity order as we obtain here.

\section{Asymptotic counting results}\label{sec:asymptotic}

We adopt here the method by Bostan, Raschel and Salvy~\cite{bostan2012} (itself relying on results by Denisov and Wachtel~\cite{denisov2015random}) to obtain asymptotic estimates for the counting coefficients of plane bipolar posets (by vertices or by edges) and transversal structures (by vertices). Let $\cS=\{SE\}\cup\{(-i,j),\ i,j\geq 0\}$ be the tandem step-set. Let $w:\mathbb{N}^2\to\mathbb{R}_+$ satisfying the symmetry property $w(i,j)=w(j,i)$. 
The induced weight-assignment on $\cS$ is $w(s)=1$ for $s=SE$ and $w(s)=w(i,j)$ for $s=(-i,j)$. 
Let $a_n^{(w)}$ be the weighted number (i.e., each walk $\sigma$ is counted with weight $\prod_{s\in \sigma}w(s)$) of quadrant tandem walks of length $n$, for some fixed starting and ending points. Let 
\[
S(z;x,y):=\frac{x}{y}z^{-2}+\sum_{i,j\geq 0}w(i,j)\frac{y^j}{x^i}z^{i+j},\]
 let $S(z):=S(z;1,1)$, and let $\rho$ be the radius of convergence (assumed here to be strictly positive) of $S(z)-z^{-2}$.   
Let $\tw(s):=\frac1{\gamma}w(s)z_0^{y(s)-x(s)}$ be the modified weight-distribution where $\gamma,z_0>0$ are adjusted so that $\tw(s)$ is a probability distribution (i.e. $\sum_{s\in\cS} \tw(s)=1$) and the drift is zero, which is here equivalent to having $z=z_0\in(0,\rho)$ solution of $S'(z)=0$ (one solves first for $z_0$ and then takes $\gamma=S(z_0)$;  
note also that $S'(z)$ is increasing on $(0,\rho)$ so that $z_0$ is unique if it exists). 
Then as shown in~\cite{denisov2015random,bostan2012} we have, for some $c>0$,
\begin{equation}
a_n^{(w)}\sim c\ \!\gamma^n\ \!n^{-\alpha},\ \mathrm{where}\ \alpha=1+\pi/\mathrm{arccos}(\xi),\ \mathrm{with}\ \xi=-\frac{\partial_x\partial_yS(z_0;x,y)}{\partial_x \partial_xS(z_0;x,y)}\Big|_{x=1,y=1}.
\end{equation}  
As shown in~\cite{denisov2015random}, the dependence of $c$ on the starting point $(i_0,j_0)$ and ending point $(i_1,j_1)$ is of the form  $c=\kappa\ \!g(i_0,j_0)g(j_1,i_1)$ 
for some absolute constant $\kappa>0$ and a function $g$ from $\mathbb{N}^2$ to $\mathbb{R}_{>0}$,
(which is a discrete harmonic function for the walk model). In the special case where the starting point is $(0,p)$ and 
ending point is $(q,0)$, the dependence of $c$ on $p$ and $q$ is thus of the form $c=\kappa\ \!f(p)f(q)$ for some function $f$ from $\mathbb{N}$ to $\mathbb{R}_{>0}$, in which case we say that $c$ has separate dependence on $p$ and $q$.  

\begin{proposition}[Plane bipolar posets counted by edges]\label{prop:asympt_posets_edges} 
For fixed $p,q\geq 1$, let $e_n(p,q)$ be the number of plane bipolar posets 
of outer type $(p,q)$ with $n$ edges. Let $z_0\approx 0.54$ be the unique positive root of $z^4+z^3-3z^2+3z-1$. Let $\gamma=5z_0^3+7z_0^2-13z_0+9\approx 4.80$, $\xi=1-z_0/2\approx 0.73$, and $\alpha=1+\pi/\mathrm{arccos}(\xi)\approx 5.14$.
Then there exists a positive constant $c$ (with separate dependence on $p$ and $q$) such that 
\[
e_n(p,q)\sim c\ \!\gamma^n\ \!n^{-\alpha}.
\]
Moreover, the associated generating function $\sum_{n\geq 1}e_n(p,q)z^n$ is not D-finite. These results also apply to the coefficients $e_n$ (number of plane bipolar posets with $n$ edges) via the relation $e_n=e_{n+4}(1,1)$ (valid for $n\geq 2$).  
\end{proposition}
\begin{proof}
This case corresponds to taking $w(i,j)=\mathbf{1}_{i\neq 0,j\neq 0}$. By Proposition~\ref{prop:weight}, 
$e_n(p,q)$ is the number of $w$-weighted tandem walks of length $n-1$ from $(0,p)$ to $(q,0)$. We have
\[S(z;x,y)=\frac{x}{y}z^{-2}+\frac{z/x}{1-z/x}\frac{zy}{1-zy}.\]
Hence $S(z)-z^{-2}=\frac{z^2}{(1-z)^2}$, of radius of convergence $\rho=1$. Next we have $S'(z)=\frac{2(z^4+z^3-3z^2+3z-1)}{z^3(1-z)^3}$; and thus $S'(z)$ has a positive root in $(0,1)$ which is the unique positive root $z_0\approx 0.54$ of the polynomial $P(z)=z^4+z^3-3z^2+3z-1$. 
Since $P(z)$ is the minimal polynomial of $z_0$, any rational expression (with coefficients in $\mathbb{Q}$) in $z_0$ reduces to a polynomial 
expression in $z_0$ of degree smaller than $\mathrm{deg}(P)=4$. Clearly, $\gamma=S(z_0)$ and $\xi=-\frac{\partial_x\partial_yS(z_0;x,y)}{\partial_x \partial_xS(z_0;x,y)}\big|_{x=1,y=1}$ have rational expressions in $z_0$ (since $S(z;x,y)$ is rational in $x,y,z$), and the expressions we obtain by reduction are $\gamma=5z_0^3+7z_0^2-13z_0+9$ and $\xi=1-z_0/2$. 

For a generating function $\sum_na_nz^n$ with $a_n\in\mathbb{Z}$, a known sufficient condition for being not D-finite is that $a_n\sim c\ \gamma^nn^{-\alpha}$ with $\alpha\notin\mathbb{Q}$. As shown in~\cite{bostan2012}, if $\xi$ is an algebraic number with minimal polynomial $X(s)$, then 
  $\alpha:=1+\pi/\mathrm{arccos}(\xi)$ is rational iff (the numerator of) $X(\tfrac1{2}(\zeta+\zeta^{-1}))$ has a cyclotomic polynomial among its prime factors. Here the minimal polynomial of $\xi$ is $X(s)=P(2-2s)$ (since $z_0=2-2\xi$), and the numerator of $X(\tfrac1{2}(\zeta+\zeta^{-1}))$ is the prime polynomial $\zeta^8-9\,\zeta^7+31\,\zeta^6-62\,\zeta^5+77\,\zeta^4-62\,\zeta^3+31\,\zeta^2-\zeta+1$ which is not cyclotomic since, as recalled in~\cite{bostan2012}, all cyclotomic polynomials of degree at most $30$ have their coefficients in $\{-2,-1,0,1,2\}$. Hence for any fixed $p,q\geq 1$, the generating function $\sum_{n\geq 1}e_n(p,q)z^n$ is not D-finite.

Finally, the claimed relation $e_n=e_{n+4}(1,1)$ follows from the fact that, for $n\geq 2$, a plane bipolar poset with $n$ edges identifies to a plane bipolar poset of outer type $(1,1)$ with $n+4$ edges, upon adding a left outer path of length $2$ and a right outer path of length $2$. 
\end{proof}

\begin{proposition}[Plane bipolar posets counted by vertices]\label{prop:asympt_posets_vertices} For fixed $p,q\geq 0$, let $b_n(p,q)$ be the number of plane bipolar posets 
of pole-type $(p,q)$ with $n+2$ vertices. 
Then there exists a positive constant $c$ (with separate dependence on $p$ and $q$) such that 
\[
b_n(p,q)\sim c\ \!\big(\tfrac{11+5\sqrt{5}}{2}\big)^n\ \!n^{-6}.
\]
These results also apply to the coefficient $b_n$ (number of plane bipolar posets with $n+2$ vertices) via the relation $b_n=b_{n+2}(0,0)$.  
\end{proposition}
\begin{proof}
This case corresponds to  $w(i,j)=\binom{i+j}{i}$. By Proposition~\ref{prop:weight}, $b_n(p,q)$ is the number of $w$-weighted tandem walks of length $n$ from $(0,p)$ to $(q,0)$. 
We have 
\[
S(z;x,y)=\frac{x}{y}z^{-2}+\frac{1}{1-z/x-zy}.
\]
Hence $S(z)-z^{-2}=1/(1-2z)$ of radius of convergence $\rho=1/2$. Next, we have $S'(z)=\frac{2(z-1)(z^2-3z+1)}{(z^3(1-2z)^2}$ thus $S'(z)$ has a positive root in $(0,1/2)$ which is the unique positive root $z_0=\frac{3-\sqrt{5}}{2}\approx 0.38$ of the polynomial $z^2-3z+1$. Any rational expression in $z_0$ thus reduces to an expression $a+a'\sqrt{5}$ with $a,a'$ in $\mathbb{Q}$.   We find 
$\gamma=\tfrac1{2}(11+5\sqrt{5})\approx 11.09$, and   
$\xi =\tfrac1{4}(1+\sqrt{5})\approx 0.81$ (we actually have $\xi=1-z_0/2$, as for plane bipolar posets counted by edges), and $\alpha=1+\pi/\mathrm{arccos}(\xi)=6$. 

Finally, the claimed relation $b_n=b_{n+2}(0,0)$ follows from the fact that for $n\geq 0$ a plane bipolar poset with $n+2$ vertices identifies to a plane bipolar poset of pole-type $(0,0)$ with $n+4$ vertices, upon creating a new sink (resp. a new source) connected to the former sink (resp. source) by a single edge. 
\end{proof}
\begin{remark}
We recover, as expected in view of the previous section, the asymptotic constants $\gamma$ and $\alpha$ for plane permutations, which were obtained in~\cite{bouvel2018semi} (where $c$ was also explicitly computed). \hfill$\triangle$
\end{remark}

\begin{theorem}[Transversal structures]\label{theo:asympt_trans_structures}
Let $v\in[0,+\infty)$. Let $\gamma(v)>1$, $\xi(v)\in(0,1)$ be given by (see Figure~\ref{fig:plots} for plots)
\[
\gamma(v)=\frac {1}{2(2+v)} \big(2v^{2}+18\,v+27+(9+4v)^{3/2}\big),\ \ \ \xi(v)=\frac {1}{4(2+v)^2} \left(4v^{2}+14v+11+
\sqrt {9+4v} \right),
\]
and let $\alpha(v)=1+\pi/\mathrm{arcos}(\xi(v))$. For fixed $p,q\geq 0$, let $t_n^{(p,q)}(v)$ be the $v$-weighted number of transversal structures of WE-type $(p,q)$ with $n+4$ vertices.  

Then there exists a positive constant $c(v)$ (with separate dependence on $p$ and $q$) such that 
\[
t_n^{(p,q)}(v)\sim c(v)\ \!\gamma(v)^nn^{-\alpha(v)}.
\]
These results also apply to the $v$-weighted number $t_n(v)$ of transversal structures with $n+4$ vertices, via the relation $t_n(v)=t_{n+2}^{(1,1)}(v)$.  

For $v=0$ (triangulated case) and $v=1$ (general case), we have
\[
\gamma(0)=\frac{27}{2},\ \xi(0)=\frac{7}{8},\ \alpha(0)\approx 7.22,\ \ \ \ \gamma(1)=\frac{47+13\sqrt{13}}{6},\ \xi(1)=\frac{29+\sqrt{13}}{36},\ \alpha(1)\approx 8.18.
\]
In both cases, the associated generating function $\sum_{n\geq 1}t_n(v) z^n$ is not D-finite.
\end{theorem}
\begin{figure}
	\center
	\begin{subfigure}[b]{0.5\textwidth}
		\center
		\includegraphics[width=\textwidth]{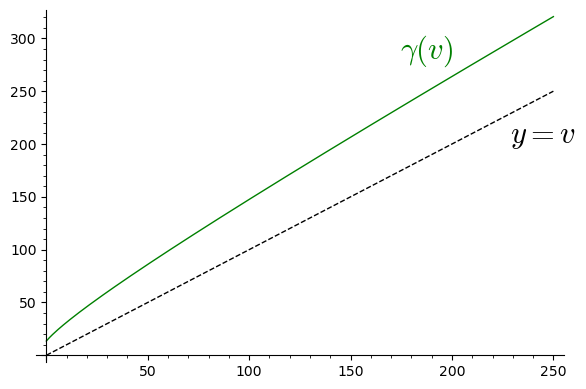}
	\end{subfigure}%
	\begin{subfigure}[b]{0.5\textwidth}
		\center
		\includegraphics[width=\textwidth]{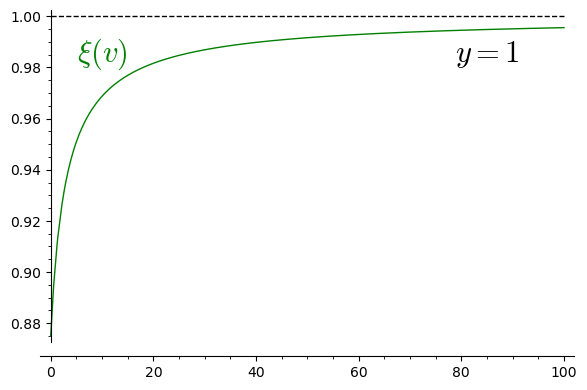}
	\end{subfigure}
	\caption{The plot of~$\gamma(v)$ (on the left) and~$\xi(v)$ (on the right) and their respective equivalents~$v$ and~$1$.}
	\label{fig:plots}
\end{figure}
\begin{proof}
For $v=0$, we have $w(i,j)=\binom{i+j-2}{i-1}$, and the associated series is 
\[
\tilde{S}(z;x,y)=x\bar y z^{-2}+\frac{z^2\bar x y}{1-z(\bar x+v\bar x y+y)}. 
\]
By Proposition~\ref{prop:weight}, $t_n^{(p,q)}(v)$ is the $w$-weighted number of tandem walks 
from $(0,p)$ to $(q,0)$ with $n$ SE steps. Consider such a walk $\pi$. Since there is no step of the form $(0,j)$ (such a step has weight $0$ here) it starts with a SE step, and since there is no step of the form $(-i,0)$, it also ends with a SE step. 
Let $\tilde{\pi}$ be the walk obtained from $\pi$ by deleting the last SE step, and aggregating the other steps into groups formed by a SE step followed by a (possibly empty) sequence of non-SE steps; precisely each such group $s_1,\ldots,s_k$ yields the aggregated step $\sigma=s_1+\cdots+s_k$ with weight $w(\sigma)=\prod_{i=1}^kw(s_i)$. The obtained weighted walk $\tilde{\pi}=(\sigma_1,\ldots,\sigma_n)$ 
starts at $(0,p)$, ends at $(q-1,1)$ (since the last SE step of $\pi$ has been deleted), and the condition that $\pi$ stays in the quadrant $\{x\geq 0; y\geq 0\}$ translates to the condition that $\tilde{\pi}$ stays in the shifted quadrant $\{x\geq 0; y\geq 1\}$ (indeed, $\pi$ stays in the quadrant iff the starting point of every SE step is in the shifted quadrant).    Note also that the series corresponding to one aggregated step is 
\[
S(z;x,y)=\frac{x\bar y\ \!z^{-2}}{1-\frac{z^2\bar x y}{1-z(\bar x+v\bar xy +y)}}.
\]
Hence $S(z)-z^{-2}=\frac{1}{1-2z-z^2}$, of radius of convergence $\rho(v)=\frac{\sqrt{v+2}}{v+1}-1$. 
The function $S(z)$, defined on $(0,\rho(v))$, is convex, and diverges as $z\to 0^+$ and $z\to\rho(v)^-$, 
hence reaches its minimum at the unique value $z=z_0(v)\in(0,\rho(v))$, where $S'(z)=0$. 
We have
\[
S'(z)=-2\,{\frac { \left( v{z}^{2}+{z}^{2}+z-1 \right)  \left( v{z}^{2}+3\,z
-1 \right) }{ \left( v{z}^{2}+{z}^{2}+2\,z-1 \right) ^{2}{z}^{3}}}
\]
The single root of the numerator in the interval $(0,\rho(v))$ is $z_0(v)=\frac{-3+\sqrt{9+4v}}{2v}$, which is one of the two roots of $v{z}^{2}+3\,z
-1$. Note that it is a regular function at $v=0$: $z_0(v)=\frac{1}{3} - \frac{1}{27}\, v + \frac{2}{243}\, v^2+O(v^3)$. We then find
$\gamma(v)=S(z_0(v))=\frac {1}{2(2+v)} \big(2v^{2}+18\,v+27+(9+4v)^{3/2}\big)$, and $\xi(v)=\frac {1}{4(2+v)^2} \left(4v^{2}+14v+11+
\sqrt {9+4v} \right)$.

Let $v\in\{0,1\}$. To show that the associated series $\sum_n t_n(v)z^n$ is not D-finite, we consider the minimal polynomial $X(s)$ of the algebraic number $\xi(v)$.
For $v=0$, we find $X(s)=8s-7$. We have 
 $X(\tfrac1{2}(\zeta+\zeta^{-1}))=(4\zeta^2-7\zeta+4)/\zeta$, whose numerator is not cyclotomic, as it is irreducible of degree $2$ with coefficients of absolute value larger than $2$. Hence, $\sum_n t_n(0)z^n$ is not D-finite according to~\cite{bostan2012}. For $v=1$, we find $X(s)=36s^2-58 s+23$, and $X(\frac1{2}(\zeta+\zeta^{-1}))=\frac1{36\,{x}^{2}}\big(36\,{x}^{4}-58\,{x}^{3}+95\,{x}^{2}-58\,x+36\big)$, whose numerator is not cyclotomic, as it is irreducible of degree~$4$ with coefficients of absolute value larger than $2$. Hence, $\sum_n t_n(1)z^n$ is not D-finite
\end{proof}

\begin{figure}
\begin{center}
\includegraphics[width=8cm]{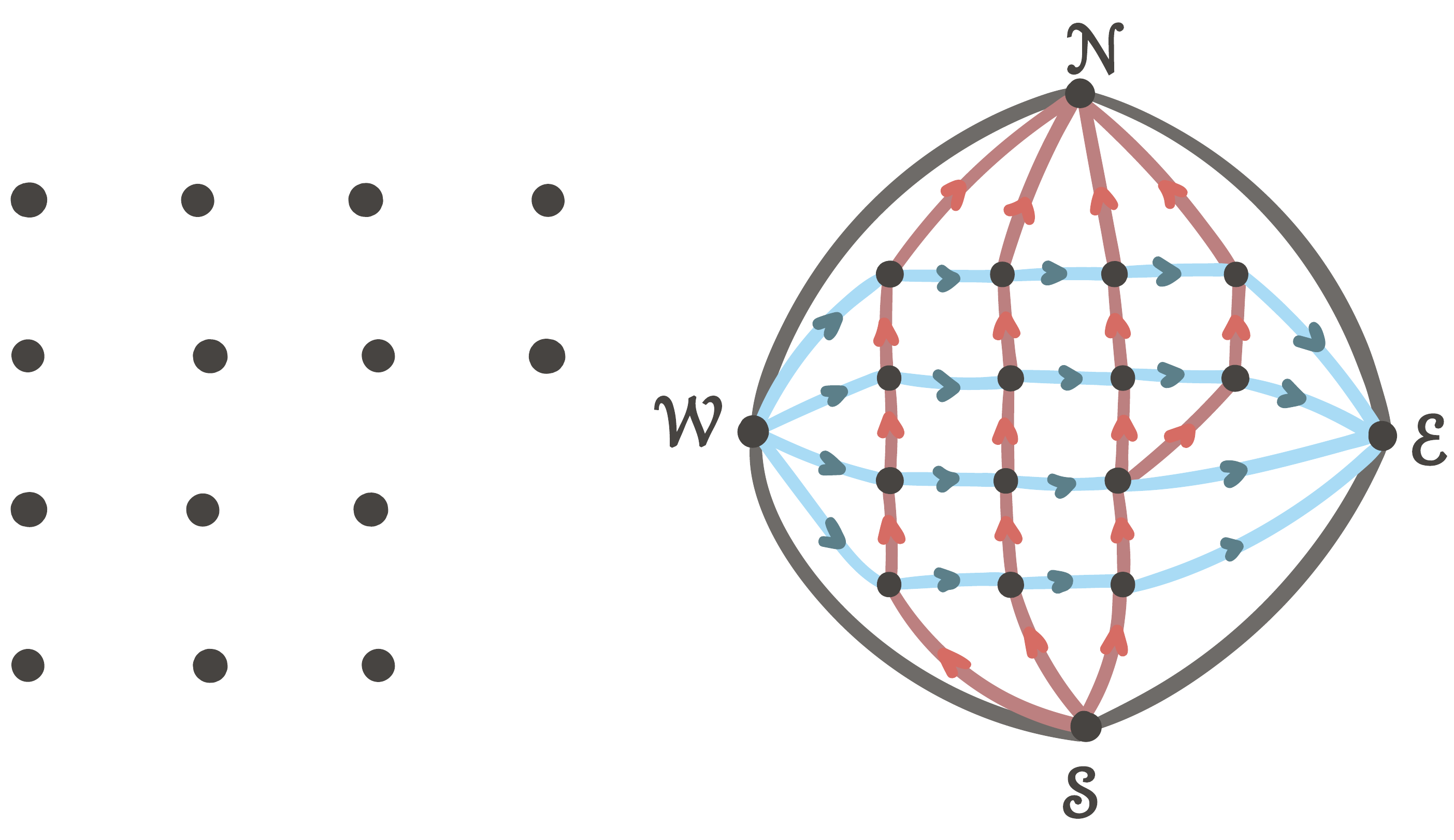}
\end{center}
\caption{Construction of a transversal structure with $n$ inner vertices and $n-\Theta(\sqrt{n})$ quadrangular faces. Letting $h=\lfloor \sqrt{n}\rfloor$, we take the $h\times h$ grid $G_h$ (regular grid with $h^2$ points), and complete it by adding the $n-h^2$ first points of the hook $G_{h+1}\backslash G_{h}$ (traversed starting with the top row left-to-right and then the right column top-to-bottom). 
We then build a transversal structure with $n$ inner vertices at the $n$ points as indicated; its number of quadrangular inner faces is $n-2h+\delta$, with $\delta=1$ for $n=h^2$, 
$\delta=0$ for $n-h^2\in[1..h]$, and $\delta=-1$ for $n-h^2\in[h+1..2h]$.}
\label{fig:grille}
\end{figure}

\begin{remark}
Our estimate implies that $t_n(0)/t_{n-1}(0)$ converges to $27/2$, which had been conjectured in~\cite{inoue2009counting} based on numerical computations.  
Another consequence of our estimate is that the coding procedure in~\cite{takahashi20094n} for triangulated transversal structures can be made asymptotically optimal,
 as it implies that the asymptotic growth rate of $t_n(0)$ is bounded by $27/2$. 
 
 Let us also mention that the recurrence for $t_n(0)$ obtained in~\cite{inoue2009counting} actually gives another quadrant walk model to obtain the asymptotic estimate for $t_n(0)$ 
 (via Denisov-Wachtel).  
Indeed, although not in a completely bijective way (simplifications occur from algebraic manipulations), the recurrence in~\cite{inoue2009counting} implies that $t_n(0)$ is the number 
of walks of length $n-1$ in the $1/8$th plane $\{0\leq y\leq x\}$, starting and ending at the origin, where the series for the set of steps is $S(x,y):=\frac1{1-y}(x+2\by+\bx\by)$. 
Equivalently, $t_n(0)$ is the number of walks of length $n-1$ in the quadrant, starting and ending at the origin, where the series for the 
set of steps is $S(x,y):=\frac1{1-\bx y}(x+2x\by+\by)$. As in our walk model, the symmetry $S(x,y)=S(\by,\bx)$ holds. Accordingly, we let $S(z;x,y):=S(x/z,yz)$,
and $S(z)=S(z;1,1)$.  We find that $S'(z)=0$ for $z=z_0:=2/3$, where we have $\gamma:=S(z_0)=27/2$ and $\xi:=-\frac{\partial_x\partial_yS(z_0;x,y)}{\partial_x \partial_xS(z_0;x,y)}\big|_{x=1,y=1}=7/8$.

Finally, we note that the recurrence obtained in~\cite{conant2014number} for the coefficients of $t_n(v)$ does not seem applicable to asymptotic enumeration, since it involves subtractions as a consequence of inclusion-exclusion. \hfill$\triangle$
\end{remark}

\begin{remark}
The case of plane bipolar posets counted by vertices, this time with $(p,q)$ corresponding to the outer type rather than the pole-type, can be treated from the weight-assignment $w(i,j)=\mathbf{1}_{i\neq 0,j\neq 0}$ (used for plane bipolar posets counted by edges) and aggregating the steps into groups formed by a SE step followed by a (possibly empty) sequence of non-SE steps. In that case the series for one aggregated step is 
\[
S(z;x,y)=\frac{x/y\ \!z^{-2}}{1-\frac{z^2y/x}{(1-z/x)(1-zy)}},
\]
which coincides with the expression $S(z;x,y)=\frac{x}{y}z^{-2}+\frac{1}{1-z/x-zy}$ (obtained from Proposition~\ref{prop:weight}) used in the proof of Proposition~\ref{prop:asympt_posets_vertices}. The asymptotic constants $\gamma=\tfrac{11+5\sqrt{5}}{2},\alpha=6$ are thus the same as those in Proposition~\ref{prop:asympt_posets_vertices}. \hfill$\triangle$
\end{remark}

We conclude this section with some observations and consequences of properties of the functions $\gamma(v)$ and $\xi(v)$. 
Let $t_{n,k}$ be the number of transversal structures with $n+4$ vertices and $k$ inner quadrangular faces. 
 Let $k(n)=\mathrm{max}(k,\ t_{n,k}\neq 0)$. Based on the Euler relation, for any transversal structure with $n+4$ vertices and $k$ inner quadrangular faces, 
we have $k=n+1-\Delta/2$, with $\Delta$ the number of triangular inner faces. Hence $k(n)\leq n-1$ (there are at least $4$ triangular faces, one incident to each outer edge). 
On the other hand, it is easy to obtain $n-k(n)=O(\sqrt{n})$, as shown in Figure~\ref{fig:grille}, hence $k(n)\sim n$. 

\begin{corollary}\label{coro:entr}
For $u_n$ any integer sequence such that $u_n\leq k(n)$ and $u_n=n-o(n)$, let $s_n=\sum_{k=u_n}^{k(n)}t_{n,k}$. Then $\lim_{n\to\infty}s_n^{1/n}=1$.  
\end{corollary}
\begin{proof}
Since $s_n\geq 1$, it suffices to show that $c:=\mathrm{lim\ sup}\ s_n^{1/n}$ satisfies $c\leq 1$. We have $t_n(v)\geq s_n v^{u_n}$. On the other hand,
 $u_n= n-o(n)$, hence $c\cdot v\leq \gamma(v)$. The expression of $\gamma(v)$ in Theorem~\ref{theo:asympt_trans_structures} ensures that $\gamma(v)\sim v$ as $v\to\infty$. Hence $c\leq 1$.  
\end{proof}

Corollary~\ref{coro:entr} agrees with the fact that in the regime where $k=n-o(n)$, there is a strong rigidity in transversal structures: they take the form of a regular grid (as in Figure~\ref{fig:grille}) with $O(n-k)=o(n)$ ``defaults" (triangular faces, or vertices of degree larger than $4$). 
Since the regime $k=n-o(n)$ has dominating asymptotic contribution to $t_n(v)$ as $v$ gets large, 
 $v$-weighted transversal structures can be proposed as a model that interpolates between a random lattice (case $v=\Theta(1)$) and a
regular lattice (case $v\to\infty$), for instance the (conjectural) local limit $L^{\infty}(v)$ is expected to converge to the regular lattice $\mathbb{Z}^2$ as $v\to\infty$. 
If we look at the subexponential term $n^{-\alpha(v)}$, since $\xi(v)$ is increasing to $1$ as $v$ increases from $0$ to $\infty$,  
$\alpha(v)=1+\pi/\mathrm{arccos}(\xi(v))$ is increasing to $+\infty$, and the corresponding central charge $c=c(v)$ (determined by $12\,(2-\alpha(v))=c-1-\sqrt{(1-c)(25-c)}$) is decreasing to $-\infty$.  
This is again consistent with the assumption that a random $v$-weighted transversal structure  should approach 
the behaviour of a regular lattice as $v$ gets large; e.g. the Hausdorff dimension is known~\cite{ding2020fractal} to increase to $2$ when $c$ decreases to $-\infty$.  

\begin{remark}
A different model of random planar lattices approaching a regular lattice (based on an auxiliary parameter) has been proposed in~\cite{kazakov1996almost,kazakov1996exact}, where the authors consider Eulerian quadrangulations
 weighted by $\beta^{n_2-4}$, with $n_2$ the number of vertices of degree $2$. When $\beta=0$, one gets a regular lattice behaviour (all vertices have degree $4$, except for four vertices of degree $2$). In~\cite{kazakov1996exact} it is demonstrated that, for any $\beta>0$, the asymptotic regime of the model is $n^{-5/2}$ (regime of ``pure gravity"). This contrasts with our model, where the polynomial exponent $\alpha$ in the asymptotic estimate varies continuously with the auxiliary parameter $v$ (in our model, when $\beta:=1/v$ tends to $0$, the exponent $\alpha$ diverges to $+\infty$, which as explained above is consistent with a regular lattice behaviour). \hfill$\triangle$    
\end{remark}

\section{Bijection between plane permutations and plane bipolar posets}\label{sec:bijec}
We give here a bijective proof of the fact, established in Proposition~\ref{prop:plane_perm}, that plane permutations of size $n$ are equinumerous to plane bipolar posets with $n+2$ vertices. 
The construction can be seen as the analog for plane bipolar posets of the bijection introduced in~\cite{bonichon2010baxter} between Baxter permutations and plane bipolar orientations. 

\subsection{Presentation of the construction and a first proof of bijectivity}
Recall that a permutation $\pi\in\frak{S}_n$ is called a plane permutation if it avoids the vincular pattern $2\underbracket[.5pt][1pt]{14}3$ (i.e., no pattern $2143$ has $1$ and $4$ adjacent).  
We adopt the standard diagrammatic representation of $\pi$, i.e., we identify $\pi$ to the set $\{(i,\pi(i)),\ i\in[1..n]\}$, whose elements are called the points of $\pi$. 
The dominance order on the points of $\pi$ is the poset where $p=(x,y)\leq p'=(x',y')$ iff $x\leq x'$ and $y\leq y'$. 
The dominance diagram of $\pi$ is obtained by drawing a segment $(p,p')$ for each pair of points such that $p'$ covers $p$ in the poset. 
It is known~\cite{bousquet2007forest} that a permutation is plane iff its dominance diagram is crossing-free.   
The completion $\opi$ of $\pi$ is the permutation on $[0,..n+1]$ where $\opi(0)=0$, $\opi(n+1)=n+1$ and $\opi(i)=\pi(i)$ for $i\in[1..n]$. Clearly, if $\pi$ is plane, then so is $\opi$. 
We denote by $\cP_n$ the set of plane permutations of size $n$, and by $\cB_n$ the set of plane bipolar posets on $n+2$ vertices.  
For $\pi\in\cP_n$, we let $\phi(\pi)$ be the dominance diagram of $\opi$, and let $\Phi(\pi)\in\cB_n$ be the underlying plane bipolar poset on $n+2$ vertices, see Figure~\ref{fig:bijection_dominance_drawing} for an example. 

\begin{figure}
\begin{center}
\includegraphics[width=14cm]{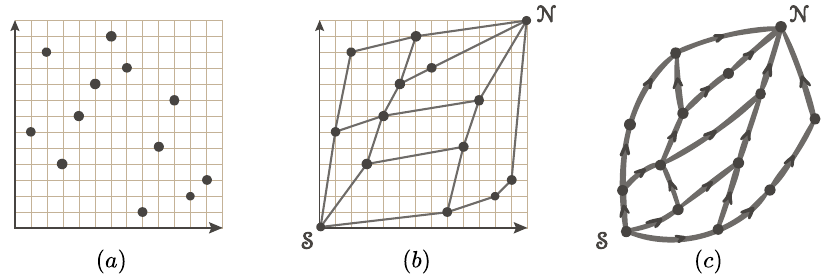}
\end{center}
\caption{(a) A plane permutation $\pi = (6, 11, 4, 7, 9, 12, 10, 1, 5, 8, 2, 3)$. (b) The embedded plane bipolar poset $\phi(\pi)$. (c) The underlying plane bipolar poset $\Phi(\pi)$.}
\label{fig:bijection_dominance_drawing}
\end{figure}

\begin{remark}
Note that the neighbors of the source $S$ of $\Phi(\pi)$ correspond to the left-to-right minima of $\pi$, the neighbors of the sink $N$ of $\Phi(\pi)$ correspond to the right-to-left maxima of $\pi$, the non-pole vertices on the right boundary of $\Phi(\pi)$ correspond to the right-to-left minima of $\pi$, and the non-pole vertices on the left boundary of $\Phi(\pi)$ correspond to the left-to-right maxima of $\pi$. \hfill$\triangle$
\end{remark}

We now describe the inverse construction $\Psi$ (illustrated in Figure~\ref{fig:def_bijection_psi}), 
\begin{figure}
	\begin{center}
	\includegraphics[width=12cm]{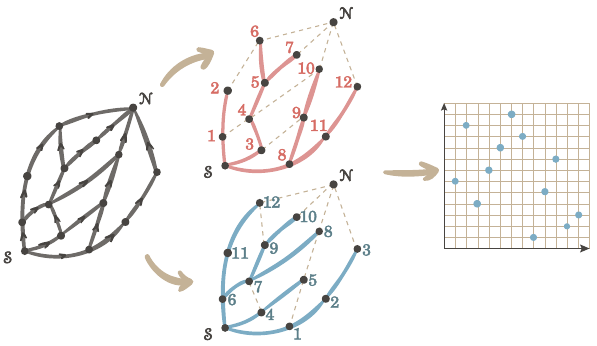}
	\end{center}
	\caption{The mapping~$\Psi$, from a plane bipolar poset to a plane permutation.}
	\label{fig:def_bijection_psi}
\end{figure}
 which is actually just the mapping formulation of a dominance drawing algorithm due to Di Battista, Tamassia and Tollis~\cite{di1992area}.   
Let $B$ be a plane bipolar poset, with $I$ its set of non-pole vertices, $S$ its source and $N$ its sink. 
The \emph{left tree} $T_\ell$ (resp. \emph{right tree} $T_r$) of $B$ is the spanning tree of $I\cup \{S\}$ where for every $v\in I$ its parent-edge is its leftmost (resp. rightmost) ingoing edge.  
For $v\in I$, we let $x(v)$ be the rank (in $[1..n]$) of $v$ for the order of first visit in clockwise order around $T_r$.  Similarly, we let $y(v)$ be the rank (in $[1..n]$) of first visit of $v$   during a counterclockwise tour around $T_\ell$. Let $\Psi(B)$ be the permutation whose point-diagram is $\{(x(v),y(v)),\ v\in I\}$, see Figure~\ref{fig:bijection_psi}.  
 Then the following property holds: 
 
\begin{figure}
\begin{center}
\includegraphics[width=10cm]{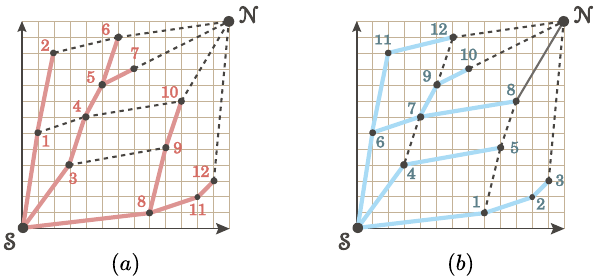}
\end{center}
\caption{The property $\Psi(\Phi(\pi))=\pi$ visualized on the embedded plane bipolar poset $\phi(\pi)$. (a) The non-pole vertices ordered by first visit in clockwise order around the right-tree $T_r$ (in red) are in order of increasing abscissas. 
(b) Likewise the non-pole vertices ordered by first visit in counterclockwise order around the left-tree $T_\ell$ (in blue) are in order of increasing ordinates.}
\label{fig:bijection_psi}
\end{figure}

\begin{claim}[\cite{di1992area}]\label{claim:PhiPsi}
For $B$ a plane bipolar poset, $\Psi(B)$ is a plane permutation and $\phi(\Psi(B))$ gives a planar drawing of $B$. 
 Hence,~$\Phi(\Psi(B))=B$. 
\end{claim}

\begin{remark}\label{rk:left-to-right}
As noted in~\cite{di1992area}, this result also follows from the study of planar lattices~\cite{kelly1975planar,kelly1982dimension}. Letting $\leq$ be the underlying partial order on the vertices of $B$, the planar embedding of $B$ yields a so-called \emph{left-to-right order} denoted $\lambda$: for $u,v\in V$ two distinct non-pole vertices, we 
let $u\ \lambda\ v$ if  $u$ is on the left of any path from the source to the sink and passing by $v$ (in which case $v$ is on the right of any path from the source to the sink and passing by $u$). The order $\lambda$ is a conjugate order for $\leq$, i.e., each pair of distinct elements of $V$  is comparable for exactly one of the two orders $\leq,\lambda$. 
 Letting $\lambda^d$ be the reverse order of $\lambda$, both 
  $\mathrm{Order}_x=\ \leq\ \cup\ \lambda$ and $\mathrm{Order}_y=\ \leq\ \cup\ \lambda^d$ are total orders on $V$, and their intersection is $\leq$, so that the poset represented by $B$  has Dushnik-Miller dimension at most $2$. 
  
  A simple argument then ensures that placing every vertex $v$ with abscissa (resp. ordinate) given by its rank in $\mathrm{Order}_x$ (resp. $\mathrm{Order}_y$), placing
  $S$ and $N$ at the lower-left and upper-right corner, and drawing edges as segments, yields a planar drawing of $B$. Moreover (see~\cite[Lemma~2]{di1992area}), this procedure coincides with the vertex placement given by $\Psi$,
  i.e.,  $\mathrm{Order}_x$ (resp. $\mathrm{Order}_y$) coincides with the order of first visit in clockwise order around $T_r$ (resp. in counterclockwise order around $T_\ell$). \hfill$\triangle$
\end{remark}

\begin{claim}\label{claim:PsiPhi}
For every $\pi\in \cP_n$, we have $\Psi(\Phi(\pi))=\pi$ (see Figure~\ref{fig:bijection_psi}). 
\end{claim}
\begin{proof}
Let $B=\Phi(\pi)$, and consider the drawing $\phi(\pi)$ of $B$. Given Remark~\ref{rk:left-to-right}, it is enough to show that for two permutation points $p,q$ ordered so that $p$ has smaller abscissa (resp. ordinate), the corresponding vertices $v_p,v_q$ of $B$ satisfy $v_p\leq v_q$ or $v_p\ \lambda\ v_q$ (resp. satisfy $v_p\leq v_q$ or $v_p\ \lambda^d\ v_q$). 
Let us check the property regarding the order on abscissas (the argument for ordinates is similar).   
If $p$ has smaller ordinate than $q$, then $v_p\leq v_q$. 
If $p$ has larger ordinate than $q$, then $p$ and $q$ are not comparable for $\leq$ (since it is the dominance order). In that case, $q$ lies in the south-east quadrant of $p$,
so that it is necessarily on the right of any path from source to sink passing by $p$ (such a path being a broken line of segments of positive slope), hence $v_p\ \lambda\ v_q$. 
\end{proof}


From Claim~\ref{claim:PhiPsi} and Claim~\ref{claim:PsiPhi} we obtain:
\begin{proposition}\label{prop:bij}
The mapping $\Phi$ is a bijection from $\cP_n$ to $\cB_n$. Its inverse is $\Psi$. 
\end{proposition}



\subsection{Generating trees for plane permutations and plane bipolar posets}

In this section, we give an alternative proof of Proposition~\ref{prop:bij} via isomorphic generating trees, similarly as the proof method in~\cite{bonichon2010baxter}.   
This also yields an alternative proof of the fact that the vertex placement by $\Psi$ yields a planar dominance drawing of any plane bipolar poset (and thus also 
an alternative proof of the fact that planar lattices have Dushnik-Miller dimension at most $2$).


We recall first the generating tree for plane permutations as introduced in~\cite{bouvel2018semi}.  
Let $n\geq 1$ and $\pi'\in \cP_{n+1}$. The \emph{parent} of $\pi'$ is the permutation $\pi$ of size $n$ obtained by deleting the rightmost element and renormalizing, i.e., it is the permutation $\pi$ such that $\pi(i)=\pi'(i)-\mathbf{1}_{\pi'(i)>\pi'(n+1)}$ for $i\in[1..n]$. Clearly, $\pi\in\cP_n$. 
Conversely, for $\pi\in \cP_{n}$ 
and $a\in[1..n+1]$, let $\pi\cdot a$ be the permutation $\pi'$ of size $n+1$ such that $\pi'(n+1)=a$, and  $\pi'(i)=\pi(i)+\mathbf{1}_{\pi(i)\geq a}$ for $i\in[1..n]$. Then $a$ is called \emph{active} for $\pi$ if $\pi\cdot a$ is in $\cP_{n+1}$, in which case the point $p$ of $\opi$ of ordinate $a-1$ is also called active and we use the notation $\pi_p$ for $\pi\cdot a$. A point $q$ in $\opi$ clearly corresponds to a point in 
$\overline{\pi_p}$, and by a slight abuse of notation this point is also referred to as $q$. The following characterizes the non-active points and will be used later on.
 
\begin{claim}\label{claim:charact_active}
For $\pi\in \cP_n$, the point $S=(0,0)$ is always active, and moreover a point $p\in\pi$ is active 
iff it does not occur as the left element of a pattern $213$ in $\pi$. 
\end{claim} 
\begin{proof}
The first statement follows from the fact that the value $a=1$ is always active. 
 Regarding the second statement, if $p$ occurs as the left element of a pattern $213$ in $\pi$, then this pattern becomes a pattern $2\underbracket[.5pt][1pt]{14}3$ in 
$\pi_p^{-1}$, so $\pi_p$ (as $\pi_p^{-1}$) is not a plane permutation. Conversely, if $p$ is not active, then $\pi_p$ is not plane, 
and neither is $\pi_p^{-1}$. Hence the addition of the new point $p'$ at the right end creates a pattern $2\underbracket[.5pt][1pt]{14}3$ in $\pi_p^{-1}$, and  $p$ is the $1$ in that pattern (i.e., the point with smallest abscissa among the $4$ points). In $\pi$ the pattern formed by the $3$ points of abscissa $\leq n$ is thus $231$ with $p$ as the left element.   
\end{proof}

Let $\pi\in P_n$. Note that the rightmost point $(n,\pi(n))$ of $\pi$ is always active. The active points of $\opi$ with ordinate weakly larger (resp. strictly smaller) than $\pi(n)$  
are called \emph{upper-active} (resp. \emph{lower-active}).     
We let $p_1,\ldots,p_{h+k}$ be the active points of $\opi$ ordered by decreasing ordinate, with $p_1,\ldots,p_h$ the upper-active ones (with $p_{h}$ the rightmost point of $\pi$),
and $p_{h+1},\ldots,p_{h+k}$ the lower-active ones (ending with $p_{h+k}=(0,0)$). 
Let $j\in[1..h+k]$ and let $\pi'=\pi_{p_j}$, which is to be the $j$th child of $\pi$ in the generating tree of plane permutations. In $\pi'$ let $p'$ be the new added point at the right end. Then as shown in~\cite{bouvel2018semi} (and a consequence of Claim~\ref{claim:charact_active}),  the upper-active points of $\overline{\pi'}$ are $p_1,\ldots,p_{j-1},p'$ (for any $j\in[1..h+k]$), while the lower-active points are $p_{h},\ldots,p_{h+k}$ if $j\leq h$, and $p_{j},\ldots,p_{h+k}$ if $j>h$. In other words~\cite[Prop.5]{bouvel2018semi}, upon assigning the label pair $(h,k)$ to a plane permutation, the generating tree of $\cP=\cup_n\cP_n$ is isomorphic to the tree generated by the succession rule
\begin{equation}\label{eq:succession_rule}
\Omega=\left\{\begin{array}{ll}
(1,1)\\
(h,k) \rightsquigarrow \hspace{-3mm}& (1,k+1), \dots ,(h-1,k+1), (h,k+1)\\
& (h+1,k), \dots , (h+k,1). \end{array}\right.
\end{equation}

We now describe a generating tree for plane bipolar posets counted by vertices. Let $n\geq 1$, and let $B'\in \cB_{n+1}$. We first define how the \emph{parent} $B$ of $B'$ (which has one vertex less) is obtained. 
Let $e'=(u',N)$ be the last edge on the right boundary of $B'$ (i.e., the rightmost ingoing edge of $N$); $e'$ is called the top-right edge of $B'$, and its origin is called the top-right vertex of $B'$.  
Let $B'/e'$ be the plane bipolar orientation obtained by contracting $e'$. Note that all the inner faces have the same type in $B'/e'$ as in $B'$, except if the face $f$ on the left of $e'$ is an inner face, in which case its type $(i,j)$ in $B'$ becomes $(i,j-1)$ in $B'/e'$. In other words $B'/e'$ is a plane bipolar poset if $f$ is the left outer face or is an inner face of type $(i,j)$ with $j\geq 2$; we let $B:=B'/e'$ in this  case (Figure~\ref{fig:parent}(i)). Note that $B$ has one edge less than $B'$. Otherwise $f$ is an inner face whose right boundary is reduced to a single edge $\epsilon$ in $B'/e'$, and we let $B$ be the plane bipolar poset $(B'/e')\backslash \epsilon$. These cases are shown in Figure~\ref{fig:parent} (ii) (when $\epsilon$ has an inner face on its right) and Figure~\ref{fig:parent}(iii) (when $\epsilon$ has the right outer face on its right). In these last two cases, $B$ has two edges less than $B'$. 

\begin{figure}
\begin{center}
\includegraphics[width=14cm]{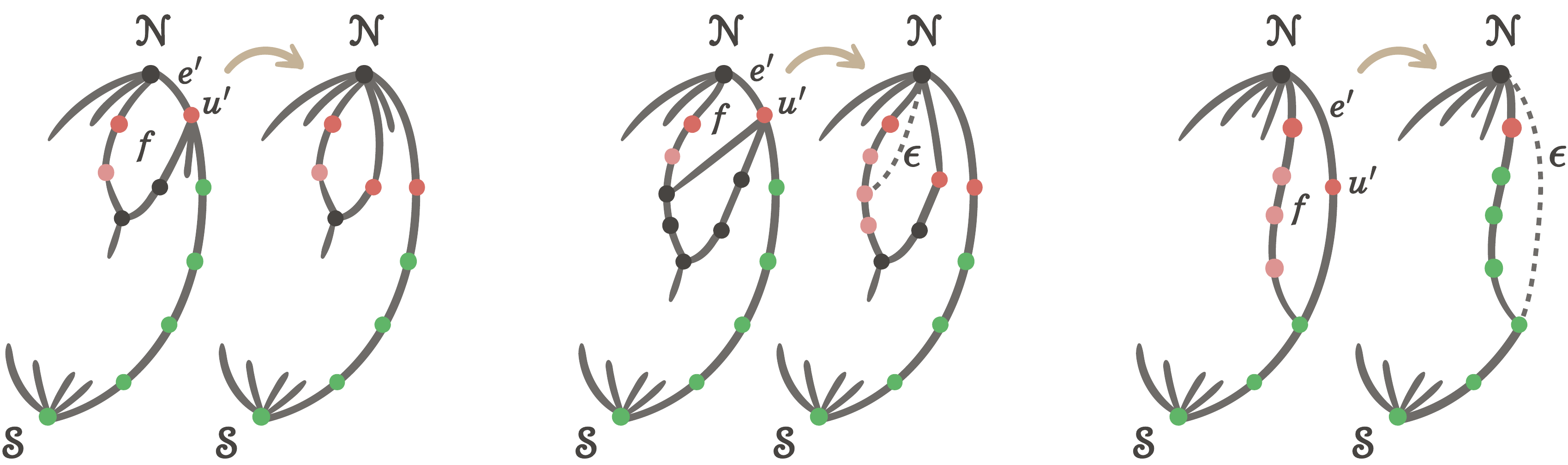}
\end{center}
\caption{The 3 possible cases for the parent of a plane bipolar poset. }
\label{fig:parent}
\end{figure} 
\begin{figure}
\begin{center}
\includegraphics[width=12cm]{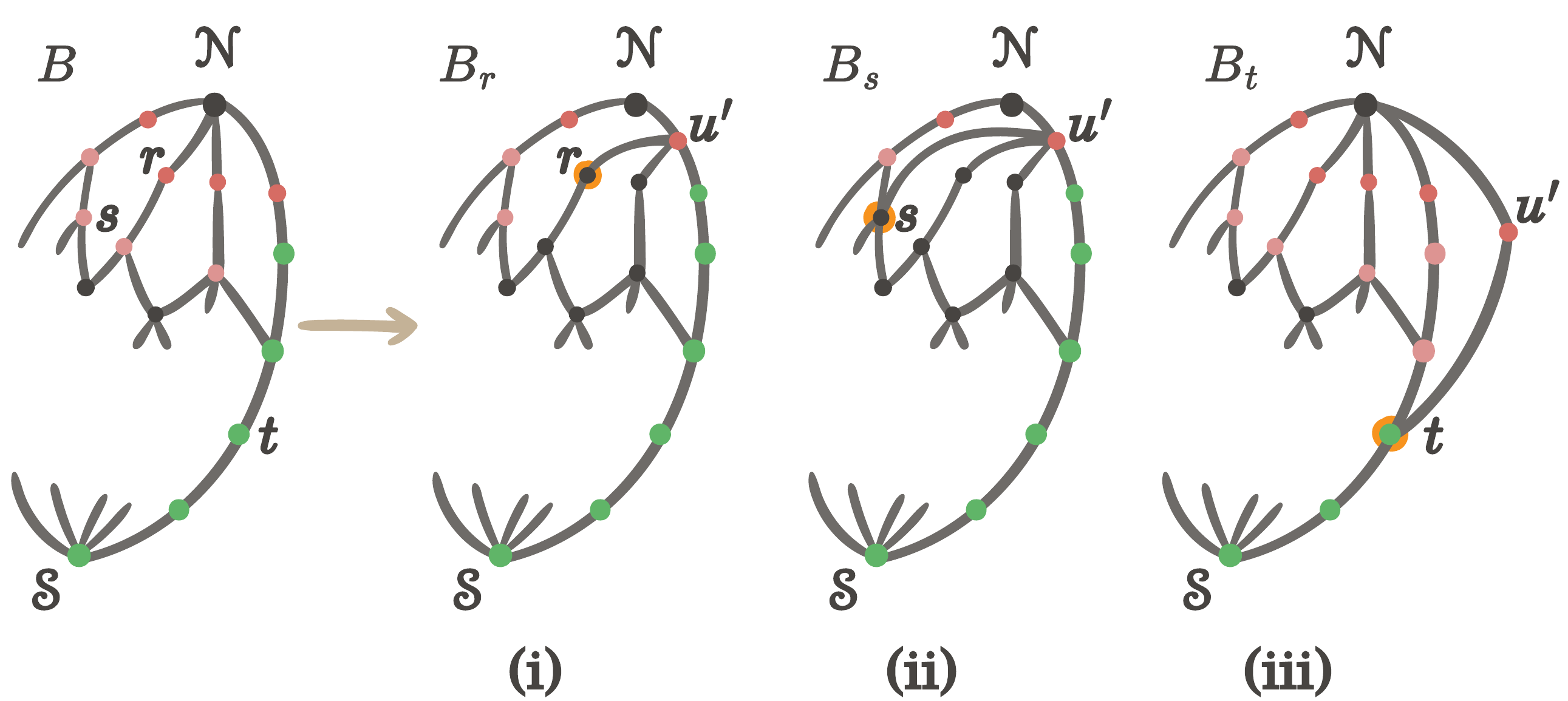}
\end{center}
\caption{The 3 types of children $B_v$ of a plane bipolar poset $B$: (i) when $v$ (here $r$) is a quasi-max, (ii) when $v$ (here $s$) is upper-active and not quasi-max, (iii) when $v$ (here $t$) is lower-active.}
\label{fig:children}
\end{figure}

Conversely, to describe the set of children of a given $B\in\cB_n$ we need a bit of terminology.  The neighbors $q_1,\ldots,q_s,q_{s+1}$ of $N$ ---ordered from left to right, with $q_{s+1}$ the top-right vertex--- are called the \emph{quasi-max} vertices of $B$, the \emph{upper faces} of $B$ are the inner faces $f_1,\ldots,f_s$ (ordered from left to right, possibly $s=0$ and this set is empty) incident to $N$, so that $f_i$ is the face on the right of the edge $(q_i,N)$ for $i\in[1..s]$.   
The \emph{upper-active vertices} of $B$ are those on the left boundary of the upper faces (excluding $N$ and the bottom vertex in each such face), plus the top-right vertex $u=q_{s+1}$. 
The \emph{lower-active} vertices of $B$ are those on the right boundary except $\{u,N\}$ (but including $S$).

As we describe now, in the generating tree there is one child $B_v$ for each active vertex $v$ of $B$ (i.e., for each vertex $v$ of $B$ that is either lower-active or upper-active).  
For $v$ a quasi-max vertex $q_i$ of $B$ (with $i\in[1..s+1]$), $B_v$ is obtained as follows (see Figure~\ref{fig:children}(i)): create a new vertex $u'$ (the new top-right vertex) whose unique out-edge points to $N$, and redirect the out-edge of each quasi-max $(q_i,\ldots,q_{s+1})$ to point to $u'$ so that, from left to right, the in-neighbors of $N$ are $q_1,\ldots,q_{i-1},u'$, and the in-neighbors of $u'$ are $q_{i},\ldots,q_{s+1}$.  
For $v$ an upper-active vertex of $B$ that is not a quasi-max, let $f_i$ ($i\in[1..s]$) be the inner face having $v$ on its left boundary. Then  $B_v$ is obtained as follows: create a new vertex $u'$ (the new top-right vertex) whose unique out-edge points to $N$, connect $v$ to $u'$ and redirect the out-edge of each quasi-max $(q_{i+1},\ldots,q_{s})$ to point to $u'$ so that, from left to right, the in-neighbors of $N$ are $q_1,\ldots,q_i,u'$, and the in-neighbors of $u'$ are $v,q_{i+1},\ldots,q_{s+1}$.  
Finally, for $v$ a lower-active vertex of $B$, $B_v$ is obtained from $B$ by adding a path of length $2$ from $v$ to $N$ on the right side. 
Note that $B_v$ has one more edge than $B$ when $v$ is a quasi-max, and two more edges otherwise. 
The following statement is readily checked case-by-case (looking at Figure~\ref{fig:parent} and Figure~\ref{fig:children}). 

\begin{claim}\label{claim:gen_tree_bipolar_poset}
For $B\in\cB_n$ and $v$ an active vertex of $B$, the parent of $B_v$ is $B$. Conversely, for $B'\in\cB_{n+1}$, let $B$ be the parent of $B'$
and let $v$ be the vertex of $B$ corresponding to the leftmost in-neighbor of the top-right vertex of $B'$. Then $v$ is active in $B$, and  $B'=B_v$. 
\end{claim}

Claim~\ref{claim:gen_tree_bipolar_poset} ensures that, for $B\in \cB_n$, the children of $B$ in the generating trees (i.e., the plane bipolar posets whose parent is $B$) are exactly the posets $B_v\in\cB_{n+1}$ with $v$ running over all active vertices of $B$.
We let $h$ (resp. $k$) be the number of upper-active (resp. lower-active) vertices of $B$, and we list the active vertices as the concatenation of $L(f_1),\ldots,L(f_s),L_{\mathrm{lower}}$, where $L(f_i)$ is the downward list of vertices on the left boundary of $f_i$ (excluding the top and bottom vertex of $f_i$), and $L_{\mathrm{lower}}$ is the list of lower-active vertices ordered downward along the right outer boundary of $B$. Let $v_1,\ldots,v_{h+k}$ be the obtained list, called the \emph{downward ordering} of active vertices (see the right-part of Figure~\ref{fig:active_vert} for an example). 
Note that $v_1,\ldots,v_h$ are the upper-active vertices and $v_{h+1},\ldots,v_{h+k}$ are the lower-active ones.   

Then one easily checks (looking at Figure~\ref{fig:children}) that the upper-active vertices of $B_{v_j}$ ---with $u'$ the new created top-right vertex--- are $v_1,\ldots,v_{j-1},u'$  (for any $j\in[1..h+k]$), while the lower-active vertices are $v_{h},\ldots,v_{h+k}$ if $j\leq h$, or are $v_{j},\ldots,v_{h+k}$ if $j>h$. Hence, upon assigning the label pair $(h,k)$ to every plane bipolar poset, the generating tree of $\cB=\cup_n\cB_n$ is isomorphic to the tree generated by the succession rule~\eqref{eq:succession_rule}. To summarize we obtain:

\begin{proposition}
The class $\cB=\cup_n\cB_n$ of plane bipolar posets is in bijection with the class $\cP=\cup_n\cP_n$ of plane permutations, via the common 
succession rule~\eqref{eq:succession_rule}. The induced (recursively specified) bijection between $\cP_n$ and $\cB_n$ is called the \emph{canonical bijection}. 
\end{proposition}

Our goal is now to show that the mapping $\Phi$ coincides with the canonical bijection. In order to do so, 
we first show that for $\pi\in\cP_n$ the 
active points of $\opi$ match the active vertices of $\Phi(\pi)$ (see Figure~\ref{fig:active_vert} for an example):

\begin{figure}
\begin{center}
\includegraphics[width=10cm]{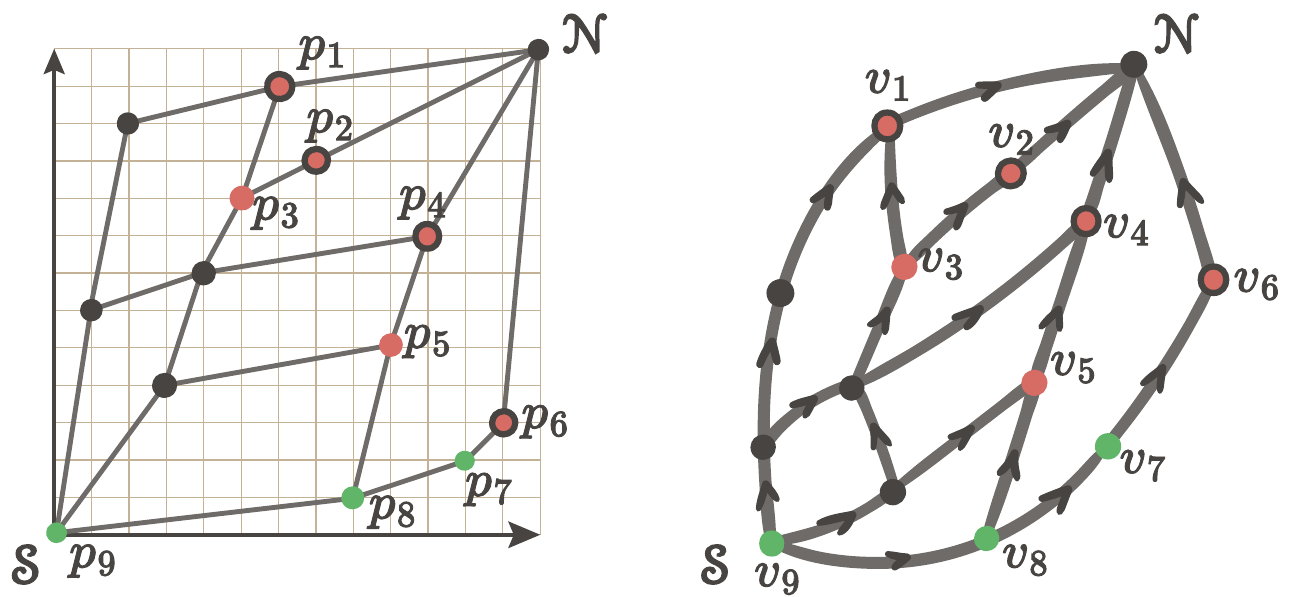}
\end{center}
\caption{For $\pi\in\cP_n$, the active points $p_1,\ldots,p_{h+k}$ of $\opi$ listed downward (i.e., by decreasing ordinate)  match the active vertices $v_1,\ldots,v_{h+k}$ of $\Phi(\pi)$ listed downward. (Upper-active points/vertices are shown red, while lower-active points/vertices are shown green.)}
\label{fig:active_vert}
\end{figure}

\begin{claim}\label{claim:active}
Let $\pi\in\cP_n$. Then the active vertices of $\phi(\pi)$ coincide with the active points of $\opi$. Moreover the downward ordering 
of the active vertices coincides with the downward ordering (ordering by decreasing ordinate) of the active points. 
\end{claim}
\begin{proof}
Consider a non-active point $p\in\pi$. By Claim~\ref{claim:charact_active} it occurs as the left element of a pattern $213$ in $\pi$, and it is easy to see that it must also occur as the left element of a pattern $2\underbracket[.5pt][1pt]{13}$ in $\pi$. Let $p',p''$ be the points corresponding to $1,3$ in the pattern, which form a rise in $\pi$, and thus form an edge $e$ in $\phi(p)$. Moreover, since $p$ is dominated by $p''$ but not by $p'$, 
the rightmost path $\gamma$ from $p$ (i.e., the path taking the rightmost outgoing edge at every step) has to hit $p''$, and it has to reach $p''$  with $e$ on the right side. Note that if the vertex $v$ corresponding to $p$ was active, then the rightmost path from $v$ would not have any ingoing edge on its right side before reaching $N$. Hence $v$ can not be active.

Conversely, let $p$ be an active point of $\pi$, and in $\phi(\pi)$ let $\gamma$ be the rightmost path from $p$. Assume that after leaving $p$ but before reaching $N$ it passes by a point $p''\neq N$ (hence $p''\in\pi$) with an ingoing edge $e=(p',p'')$ on the right side. Since we consider the rightmost path from $p$, the point $p'$ can not dominate $p$. Since $\phi(\pi)$ is an embedded poset (no transitive edge), the point $p'$ can not be dominated by $p$ either. Hence 
the points $p,p',p''$ form a pattern $213$, with $p$ or $p'$ as the left element and with $p''$ as the right element. Geometrically one checks that if the left element was $p'$ then $e$ would have to reach $\gamma$ from the left, a contradiction. Hence the left element is $p$, contradicting the fact that $p$ is active. Hence $\gamma$ has no ingoing edge on its right side before reaching $N$, and this ensures that the corresponding vertex is active in the bipolar poset (the rightmost path follows the left side of a face incident to $N$). 

It remains to check that the downward ordering of active vertices corresponds to the ordering by decreasing ordinates. 
The quasi-max $q_1,\ldots,q_{s+1}$ of the poset, ordered from left to right, correspond to the right-to-left max of $\pi$, ordered by increasing abscissa and decreasing ordinate (see Figure~\ref{fig:active_vert}, where these are surrounded). Letting $f_1,\ldots,f_s$ be the upper faces of the bipolar poset, ordered from left to right, for $i\in [1..s]$ a vertex on the left boundary of $f_i$ is dominated by $q_i$ but not by $q_{i+1}$, hence all these vertices have ordinate between those of $q_{i+1}$ and $q_i$. And the lower-active vertices are dominated by $q_{s+1}$, hence have ordinate smaller than the ordinate of $q_{s+1}$. These observations ensure that the two orderings coincide. 
\end{proof}

\begin{figure}
\begin{center}
\includegraphics[width=10cm]{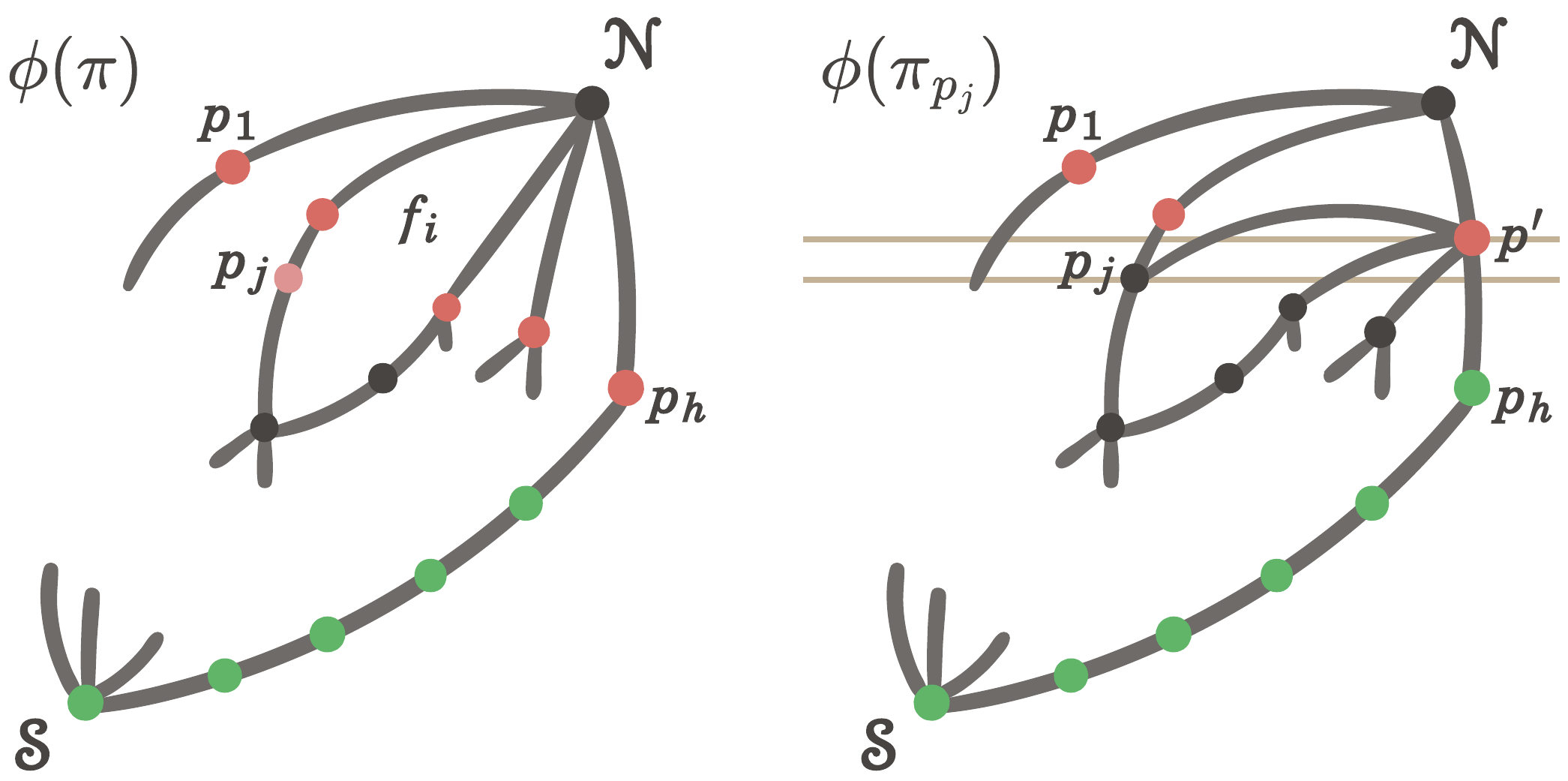}
\end{center}
\caption{The situation in the proof of Theorem~\ref{theo:Phi_canonical}.}
\label{fig:Phi_canonical}
\end{figure}

\begin{theorem}\label{theo:Phi_canonical}
The mapping $\Phi$ realizes the canonical bijection between $\cP=\cup_n\cP_n$ and $\cB=\cup_n\cB_n$.
\end{theorem}
\begin{proof}
We assume that the statement holds at size $n$ and want to prove it at size $n+1$. 
Let $\pi\in\cB_n$ having $h$ upper-active points and $k$ lower-active points, and let $p_1,\ldots,p_{h+k}$ be the active vertices of $\opi$ listed 
downward. By the inductive assumption, $B:=\Phi(\pi)$ is the image of $\pi$ by the canonical bijection (i.e., $\pi$ and $B$ are at the same place in the respective generating trees). 
By Claim~\ref{claim:active}, $p_1,\ldots,p_{h+k}$ correspond to the active vertices $v_1,\ldots,v_{h+k}$ of $B$, listed  downward. It remains to check that, for $j\in[1..h+k]$ we have $\Phi(\pi_{p_j})=B_{v_j}$. 
We treat here the case where $v_j$ is upper active but not quasi-max (the other cases can be treated similarly). Let $q_1,\ldots,q_{s+1}$ be the quasi-max vertices of $B$ ordered from left to right (corresponding to the right-to-left maxima of $\pi$ ordered downward), and let $f_1,\ldots,f_s$ be the upper faces of $B$ (ordered from left to right). 
Let $f_i$ be the one having $v_j$ on its left boundary. Let $p'$ be the added point in $\pi_{p_j}$ 
(just above $p_j$, at the right end, see Figure~\ref{fig:Phi_canonical}). To have $B_{v_j}=\Phi(\pi_{p_j})$ we just have to check that the in-neighbours of $p'$ in $\phi(\pi_{p_j})$ are $p_j,q_{i+1},\ldots,q_{s+1}$. Denoting by $y'$ the ordinate of $p_j$, this amounts to check that $p_j,q_{i+1},\ldots,q_{s+1}$ are the right-to-left maxima of $\pi|_{y\leq y'}$, which easily follows from Claim~\ref{claim:active} and the fact that $q_{i+1},\ldots,q_{s+1}$ are already right-to-left maxima of $\pi$, see Figure~\ref{fig:Phi_canonical}.   
\end{proof}

\section{Conclusion and open questions}
We have obtained exact and asymptotic enumeration results for two types of oriented planar maps: plane bipolar posets (counted by edges or by vertices), 
and transversal structures, where we can control the number of quadrangular faces (degenerate vertices in the dual rectangular tiling). 
We conclude with a list of related open questions:
\begin{itemize}
\item
Is there a direct bijection between plane bipolar posets with $n$ edges and quadrant excursions of length $n-1$ with step-set $\{0,E,S,NW,SE\}$? (These sets are 
equinumerous, by Proposition~\ref{prop:small_step_plane_posets_edges}.)
\item
Some ``Baxter-like" summation formulas are known~\cite{bouvel2018semi} for the number of plane permutations of size $n$. Is there a bijective encoding (e.g. by a system of non-intersecting lattice walks) of plane permutations or size $n$, or of plane bipolar posets with $n+2$ vertices, that yields such a formula?
\item
Is there a characterization of 4-outer maps that admit a transversal structure? As we have seen in Lemma~\ref{res:T:acyclic_poset}, 
a necessary condition is that inner face degrees are in $\{3,4\}$, and it is also quite easy to show the necessity of being 4-connected 
upon adding a vertex $v_{\infty}$ connected to the 4 outer vertices. 
We recall that a precise characterization is known~\cite{he1993finding} 
when all inner faces have degree $3$ (the condition of existence is that any triangle bounds a face).
\item
Can the recurrence for triangulated transversal structures given in~\cite{inoue2009counting} be extended to have a weight $v$ per quadrangular inner face, 
and yield another derivation of the asymptotic estimate in Theorem~\ref{theo:asympt_trans_structures}? 
\item
For fixed $v>0$, is there a local limit (as $n\to\infty$) for random transversal structures with $n+4$ vertices and weight $v$ per quadrangular inner face (where a 
random vertex is chosen as the center of the considered neighbourhoods)?
\end{itemize}


\section*{Acknowledgements} The authors thank the two anonymous referees for helpful comments to improve the presentation.   
\'E.F. and G.S. are partially supported by the project ANR-16-CE40-0009-01 (GATO) and the projet ANR-20-CE48-0018 (3DMaps), \'E.F. is also partially supported by the project ANR19-CE48-011-01 (COMBIN\'E).  

\bibliographystyle{plain}
\bibliography{ec}




\end{document}